\documentclass[twoside,11pt]{amsart}

\usepackage[utf8]{inputenc}
\usepackage[english]{babel}
\usepackage[fixlanguage]{babelbib}
\usepackage{amssymb, amsfonts, amsmath, amsthm, amscd, mathtools, thmtools, tikz, xcolor, hyperref}
\usepackage[numbers,square]{natbib}
\usepackage[all]{xy}
\usepackage{ a4wide, xfrac, calc, indentfirst, cleveref, aligned-overset, easyReview, enumitem, float, tikz-cd, microtype, tensor, epsfig, stackengine, scalerel, wasysym, color, graphicx }

\newcounter{contador}
\numberwithin{contador}{section}

\newtheorem{theorem}[contador]{Theorem}
\newtheorem{prop}[contador]{Proposition}
\newtheorem{lemma}[contador]{Lemma}
\newtheorem{corollary}[contador]{Corollary}

\theoremstyle{definition}
\newtheorem{defi}[contador]{Definition}
\newtheorem{obs}[contador]{Remark}
\newtheorem{exe}[contador]{Example}

\newcommand{\cS}{\mathcal{S}}
\newcommand{\T}{\mathcal{T}}

\newcommand{\Not}[2][0]{	
	\setcounter{enumi}{#1}
	\renewcommand{\theenumi}{#2\arabic{enumi}}
	\renewcommand{\labelenumi}{(\theenumi)}
	\setlength{\itemindent}{\widthof{#2}}
	\setlength{\itemsep}{4pt}
}

\title{The Szendrei Expansion of Restriction Semigroupoids}
\author[Haag, Lautenschlaeger and Tamusiunas]{Rafael Haag, Wesley G. Lautenschlaeger and Thaísa Tamusiunas$^*$}
\address{Instituto de Matem\'{a}tica, Universidade Federal do Rio Grande do Sul,  Av. Bento Gon\c{c}alves, 9500, 91509-900. Porto Alegre-RS, Brazil}
\email{rafaelpetasny@gmail.com}
\email{wesleyglautenschlaeger@gmail.com}
\email{thaisa.tamusiunas@gmail.com}
\thanks{$^*$ Corresponding author}
\date{} 

\begin{document}

    
    \begin{abstract}
        We introduce the concept of a restriction semigroupoid \(\cS\), which unifies the notion of restriction semigroups and restriction categories within a single structure. We prove a representation theorem, showing that every restriction semigroupoid can be embedded into a {certain} category of partial maps. Furthermore, we construct the Szendrei expansion \(Sz(\cS)\) of \(\cS\) and establish that each premorphism between two restriction semigroupoids \(\cS\) and \(\T\) is uniquely factorized by a morphism between the Szendrei expansion \(Sz(\cS)\) and \(\T\).
    \end{abstract}

    \maketitle

    \section{Introduction}

    Semigroupoids were first defined by B. Tilson \cite{tilson1987categories} in the context of the monoid decomposition theory. Roughly speaking, a Tilson semigroupoid is a category which may fail to have identities. A more general notion of semigroupoids is due to R. Exel \cite{exel2011semigroupoid}, introduced in the study of $C^\ast$-algebras. In this case, a semigroupoid is a set endowed with a partially defined binary operation which is associative. In \cite{cordeiro2023etale}, L. G. Cordeiro proved that Tilson semigroupoids are precisely the Exel semigroupoids that admit an underlying directed graph structure.

    The class of semigroupoids includes two important classes of algebraic structures: semigroups, where composition is total, and categories, where composition is partial but each object has an identity. This viewpoint allows results and constructions to be formulated once and applied across both settings. Moreover, it not only unifies existing theories but also extends their applicability to a wider range of examples. It provides a natural home for objects that arise in applications, including free semigroupoids \cite{KribsPower2004}, Markov semigroupoids \cite{exel2011semigroupoid}, and étale inverse semigroupoids \cite{cordeiro2023etale}.

    Against this backdrop, the notion of a left restriction structure can be viewed as a weakening of an inverse property and fits naturally within the semigroupoid framework. In the specific case of restriction semigroupoids, it generalizes both left restriction semigroups and left restriction categories. On the one hand, left restriction semigroups appeared as a generalization for left PP monoids used in the homological classification of semigroups \cite{fountain1977}. In some works, they are also referred to as \emph{weakly left $E$-ample semigroups}. The class of left restriction semigroups includes, for example, the semigroup of partial left maps \( \mathcal{PT}^{op}(X) \) of a set \( X \), as well as all inverse semigroups. Further details on left restriction semigroups can be found in C. Hollings' survey \cite{hollings2009PP}. On the other hand, left restriction categories arose from dominical categories, introduced to achieve an algebratization of Gödel's generalized incompleteness theorem, in the context of theoretical computer science. Several examples of restriction categories can be found in the work of J. Cockett and S. Lack \cite{cockett2002restriction}. By defining left restriction semigroupoids, we also encompass inverse semigroupoids, which, in general, may not fall within the scope of either left restriction semigroups or categories. Some properties concerning inverse semigroupoids and its intrinsic relationship with groupoids can be found in \cite{tamusiunas2023inverse}.

    A central topic in the study of abstract algebraic structures is their realization through more concrete models. In the context of inverse semigroups, V. Wagner \cite{wagner1952semigroups} and G. Preston \cite{preston1954semigroups} proved, independently, that every inverse semigroup can be embedded in some inverse semigroup $\mathcal{J}(X)$ of partial bijections of a set $X$. This result is known as the Wagner-Preston Theorem for inverse semigroups. The theorem has been generalized in two directions: to left restriction semigroups, where it has been shown that every such semigroup can be embedded into a left restriction semigroup $\mathcal{PT}(X)^{op}$ of partial left maps of a set $X$ (see, for instance, \cite[Theorem 4.12]{hollings2009PP}); and to inverse semigroupoids, proving that every inverse semigroupoid can be embedded into an inverse category $\mathcal{J}^{op}(\pi)$ of partial bijections between fibers of a function $\pi \colon X \to Y$ (see \cite[Theorem 2.32]{cordeiro2023etale}).

    In this work, we present a representation theorem that unifies the cases of restriction semigroups and inverse semigroupoids within a single framework. Specifically, we define a left restriction category $\mathcal{PT}^{op}(\pi)$ of partial left maps between the fibers of a function $\pi \colon X \to Y$, and establish the following representation theorem for left restriction semigroupoids, which will be proved in Section 3.

    \vspace{0.5cm} 

    \noindent\textbf{Theorem A.} Let $\cS$ be a left restriction semigroupoid. Then there is a restriction embedding $\cS \to \mathcal{PT}^{op}(\pi)$, for some function $\pi \colon X \to Y$.

    \vspace{0.5cm} 

    The main result of this paper is the construction of a Szendrei expansion for left restriction semigroupoids. The Szendrei expansion was defined by M. Szendrei in \cite{szendrei1989note} as a characterization of J. Birget and J. Rhodes {expansion} for semigroups, when the given semigroup is a group. Such characterization became relevant in the theory of partial actions when R. Exel \cite{exel1998partial} and J. Kellendonk and M. Lawson \cite{lawson2004actions} observed that unitary premorphisms from a group $G$ to an inverse semigroup $S$ are in correspondence with monoid morphisms from the Szendrei expansion $Sz(G)$ of $G$, which is itself an inverse semigroup, to $S$. This kind of information allows us to understand the behavior of partial actions in terms of global ones, thereby connecting partial and global theory.
    
    In the context of restriction semigroupoids, we prove that every premorphism can be uniquely factorized by a restriction morphism through its Szendrei expansion, as stated in the theorem below. The proof will be presented in Section 4. {In particular,} this result is valid for restriction categories, complementing the investigation of Cockett and Lack in \cite{cockett2002restriction}.

    \vspace{0.5cm} 

    \noindent\textbf{Theorem B.} Let $\cS$ and $\T$ be left restriction semigroupoids. If $\varphi : \cS \to \T$ is a premorphism, then there is a unique restriction morphism $\overline{\varphi} : Sz(\cS) \to \T$ such that $\varphi = \overline{\varphi} \iota$. That is, the following diagram commutes:
        
        \begin{center}
            \begin{tikzpicture}
                \tikzstyle{every path}=[draw, ->];
        
                \node (S) at (0,0) {$\mathcal{S}$};
                \node (Sz) at (0,-2) {$Sz(\mathcal{S})$};
                \node (T) at (2,0) {$\mathcal{T}$};
        
                \path (S) to node[above]{$\varphi$} (T);
                \path (S) to node[left]{$\iota$} (Sz);
                \path[dashed] (Sz) to node[below right]{$\overline{\varphi}$} (T);
            \end{tikzpicture}
        \end{center}
        Reciprocally, if $\overline{\varphi} : Sz(\cS) \to \T$ is a restriction morphism, then $\varphi = \overline{\varphi}\iota$ is a premorphism.

    \vspace{0.5cm} 

   The above theorem generalizes the correspondences \cite[Theorem 2.4]{lawson2004actions} for groups, \cite[Theorem 6.17]{lawson2006expansions} for inverse semigroups, \cite[Theorem 5.2]{gould2009partial} for left restriction semigroups, and \cite[Theorem 3.3]{tamusiunas2023inverse} for groupoids. Furthermore, when the corresponding restriction morphism and premorphism are unitary, it also recovers \cite[Theorem 4.1]{hollings2007monoids} for monoids. {We also note that Theorem B, together with the ESN Theorem for restriction semigroupoids \cite[Theorem 4.7]{lrspgesn}, can be used to generalize \cite[Proposition 4.6]{gilbert2005actions} for ordered groupoids and \cite[Theorem 5.1]{gould2011actions} for inductive constellations.}
   
    This paper is structured in the following way: in the second section, we recall the definitions for semigroupoids and graphed semigroupoids. We introduce left restriction semigroupoids and prove that every such semigroupoid is a graphed semigroupoid. In the third section, we present a representation theorem for left restriction semigroupoids, proving that every such semigroupoid can be embedded in a left restriction category of partial maps between fibers of a function. In particular, this representation theorem provides a generalization for the Wagner-Preston Theorem for inverse semigroups \cite{wagner1952semigroups}. In Section 4, we present the construction of a Szendrei expansion $Sz(\cS)$ of a left restriction semigroupoid $\cS$, and prove that premorphisms from $\cS$ to a left restriction semigroupoid $\mathcal{T}$ are in correspondence with restriction morphisms from $Sz(\cS)$, which is also a left restriction semigroupoid, to $\mathcal{T}$. In the final section, we prove that $Sz(\cS)$ is, in fact, an expansion in the sense of \cite[\textsection 2]{bierget1984expansions}. This justifies referring to it as the Szendrei \emph{expansion}. 
    
    \section{Left restriction semigroupoids}

    In this section, we recall the definitions of semigroupoids, graphed semigroupoids, and categorical semigroupoids. Then, we introduce the notion of left restriction semigroupoids and prove that every left restriction semigroupoid is categorical.
    
      \begin{defi} \label{def:semigroupoid} \cite[Definition 2.1]{exel2011semigroupoid}
            A \emph{semigroupoid} is a triple $\cS = (\cS,\cS^{(2)},\star)$ such that $\cS$ is a set, $\cS^{(2)}$ is a subset of $\cS \times \cS$, and $\star \colon \cS^{(2)} \to \cS$ is an operation which is associative in the following sense: if $r,s,t \in \cS$ are such that either
            \begin{enumerate} \Not{s}
                \item $(s,t) \in \cS^{(2)}$ and $(t,r) \in \cS^{(2)}$, or \label{s1}
                \item $(s,t) \in \cS^{(2)}$ and $(s \star t,r) \in \cS^{(2)}$, or \label{s2}
                \item $(t,r) \in \cS^{(2)}$ and $(s,t \star r) \in \cS^{(2)}$, \label{s3}
            \end{enumerate}
            then all of $(s,t)$, $(t,r)$, $(s \star t,r)$ and $(s,t \star r)$ lie in $\cS^{(2)}$, and $(s \star t) \star r = s \star (t \star r)$.
        \end{defi}
    
    At times, we will refer to \((s,t) \in \cS^{(2)}\) as ‘\(st\) is defined’, omitting the composition symbol and using concatenation to represent it. Additionally, we will use the notation \((\cS, \star)\) for a semigroupoid and, when there is no risk of confusion regarding the operation, simply \(\cS\). Semigroups and categories are particular examples of semigroupoids. For further details and illustrative examples of semigroupoids, readers are referred to \cite{cordeiro2023etale} and \cite{haagtamusiunas1}.\\
    
    
    A \emph{directed graph} is a quadruple $(G_0, G_1, D, R)$, where $G_0$ is the set of vertices, $G_1$ is the set of directed edges (which we call \emph{arrows}) and $D, R : G_1 \to G_0$ are the domain and range functions.
    
    \begin{defi} \cite[Appendix B]{tilson1987categories}
        A \emph{graphed semigroupoid} (also called \emph{Tilson semigroupoid}) is a quintuple $(G_0, G_1, D, R, \star)$, where:
        \begin{enumerate} \Not{gs}
            \item $(G_0, G_1, D, R)$ is a directed graph; \label{gs1}
        
            \item $\star$ is a function that, for each pair of arrows $(g, h)$ with $D(g) = R(h)$, associates an arrow $\star(g, h) := gh$ with $D(gh) = D(h)$ and $R(gh) = R(g)$; \label{gs2}
        
            \item $\star$ is associative, in the sense that if $g, h, k \in G_1$ are such that $D(g) = R(h)$ and $D(h) = R(k)$, then $(gh)k = g(hk)$ holds. \label{gs3}
        \end{enumerate}
    \end{defi}
    
    Not every semigroupoid is a graphed one. In fact, in \cite{exel2008inverse} Exel defined a class of semigroupoids called \emph{categorical semigroupoids}, and Cordeiro proved in \cite{cordeiro2023etale} that this class is precisely the class of graphed semigroupoids, as we will state it in the following.
    
    \begin{defi} \cite[Definition 19.1]{exel2008inverse}
        Let $\cS$ be a semigroupoid. Given $s \in \cS$, we set:
        \begin{align*}
            \cS^s = \{ t \in \cS : (s,t) \in \cS^{(2)}\} \quad \text{and} \quad \cS_s = \{ t \in \cS : (t,s) \in \cS^{(2)}\}.
        \end{align*}
    
        We say that $\cS$ is a \emph{categorical semigroupoid} if for all $s,t \in \cS$, the sets $\cS^s$ and $\cS^t$ are either equal or disjoint. 
    \end{defi}
    
    Equivalently, $\cS$ is categorical if, for all $s,t \in \cS$, the sets $\cS_s$ and $\cS_t$ are either equal or disjoint, as stated in \cite[Proposition 2.14]{cordeiro2023etale}. We say that a semigroupoid $(S,S^{(2)},\star)$ can be graphed if there is a set $S_0$ and functions $D,R \colon S \to S_0$ such that $(S_0,S,D,R,\star)$ is a graphed semigroupoid and $S^{(2)} = \{ (s,t) \colon D(s) = R(t) \}$.
    
    \begin{theorem} \cite[Theorem 2.5]{cordeiro2023etale} \label{cord24}
        Every categorical semigroupoid can be graphed. Conversely, every graphed semigroupoid is categorical.
    \end{theorem}
    
    In the following, we will introduce the concept of left and right restriction semigroupoids. First, we recall some basic definitions. Let $\cS$ be a semigroupoid. We say that $f \in \cS$ is an \emph{idempotent} if $ff$ is defined and $ff = f$. We denote by $E(\cS)$ the set of idempotents of $\cS$.  A \emph{right identity} (resp., \emph{left identity}) for an element $s \in \cS$ is an element $e \in E(\cS)$ such that $se=s$ (resp., $es=s$). Observe that not every semigroupoid has an idempotent element, and consequently, an identity element.\\
    
    A \emph{regular semigroupoid} $\cS$ is a semigroupoid in which, for all $x \in \cS$, there is $y \in \cS$ such that $(x,y), (y,x) \in \cS^{(2)}$ and $xyx = x$. If $yxy = y$ also holds, we say that $y$ is an \emph{inverse} of $x$. If the element $y$ satisfying these conditions is unique, we denote it by $y = x^{-1}$, and say that $\cS$ is an \emph{inverse semigroupoid}. 
    
    \begin{defi} \label{defsemrestesq}
        Let $\cS$ be a semigroupoid equipped with a unary operation $+ : \cS \to \cS$. We say that $\cS$ is a \emph{left restriction semigroupoid} if the following conditions are valid:
        \begin{enumerate}\Not{lr}
            \item For all $s \in \cS$, $(s^+,s) \in \cS^{(2)}$ and $s^+s = s$. \label{lr1}
            
            \item $(s^+,t^+) \in \cS^{(2)}$ if and only if $(t^+,s^+) \in \cS^{(2)}$, and in this case $s^+ t^+ = t^+ s^+$; \label{lr2}
            
            \item If $(s^+,t) \in \cS^{(2)}$, then $(s^+t)^+ = s^+t^+$.  \label{lr3}
            
            \item If $(s,t) \in \cS^{(2)}$, then $st^+ = (st)^+s$. \label{lr4}
        \end{enumerate}
    \end{defi}
    
    \begin{obs}
        \begin{enumerate}[left=0pt] \Not{2.6.}
            \item Observe that $(s,t) \in \cS^{(2)}$ if and only if $(s,t^+) \in \cS^{(2)}$. Indeed, assume that $(s,t) \in \cS^{(2)}$. By \eqref{lr1}, $t=t^+t$, hence $(s,t^+t) \in \cS^{(2)}$. By \eqref{s3}, $(t^+,t), (s,t^+t) \in \cS^{(2)}$ implies $(s,t^+) \in \cS^{(2)}$. The reverse follows by \eqref{s1} and an analogous argument. \label{261}
    
            \item The compositions in \eqref{lr3} and \eqref{lr4} are well defined. Indeed, in \eqref{lr3}, since we assume that \((s^+,t) \in \cS^{(2)}\), it follows from \eqref{261} that $(s^+,t^+) \in \cS^{(2)}$.  Similarly, in \eqref{lr4}, given that \((s,t) \in \cS^{(2)}\), we again have \((s,t^+) \in \cS^{(2)}\) by \eqref{261}. Furthermore, since $((st)^+,st) \in \cS^{(2)}$, it follows from \eqref{s3} that $((st)^+,s) \in \cS^{(2)}$. \label{262}    
        \end{enumerate}
    \end{obs}
    
    Right restriction semigroupoids are defined analogously.
    
    \begin{defi}
        Let $\cS$ be a semigroupoid equipped with a unary operation $\ast : \cS \to \cS$. We say that $\cS$ is a \emph{right restriction semigroupoid} if the following conditions are valid:
        \begin{enumerate}\Not{rr}
            \item For all $s \in \cS$, $(s,s^\ast) \in \cS^{(2)}$ and $ss^\ast = s$. \label{rr1}
            
            \item $(s^\ast,t^\ast) \in \cS^{(2)}$ if and only if $(t^\ast,s^\ast) \in \cS^{(2)}$, and in this case $s^\ast t^\ast = t^\ast s^\ast$; \label{rr2}
            
            \item If $(s,t) \in \cS^{(2)}$, then $(st^\ast)^\ast = s^\ast t^\ast$.  \label{rr3}
            
            \item If $(s,t) \in \cS^{(2)}$, then $s^\ast t = t(st)^\ast$. \label{rr4}
        \end{enumerate}
    \end{defi}

    \begin{defi}
        A (two-sided) \emph{restriction semigroupoid} is a triple $(\cS,+,\ast)$ in which $(\cS,+)$ is a left restriction semigroupoid, $(\cS,\ast)$ is a right restriction semigroupoid and that also satisfies the compatibility relation
        \begin{align*}
            (s^\ast)^+ = s^\ast \quad \text{and} \quad (s^+)^\ast = s^+,
        \end{align*}
        for all $s \in \cS$.
    \end{defi}
    
    \begin{exe}
        Every inverse semigroupoid $\cS$ is a restriction semigroupoid with right and left restrictions given, respectively, by $s^\ast = s^{-1}s$ and $s^+ = ss^{-1}$, for $s \in \cS$.
    \end{exe}
    
    \begin{obs} \label{rem_right_left}
        In this work, we will deal with left restriction semigroupoids; however, the results remain valid for right restriction semigroupoids with minor adjustments. Specifically, we define the operation \(\cdot_{op}\) on \(\cS\) by stating that \(s \cdot_{op} t\) is defined if and only if \((t, s) \in \cS^{(2)}\), and in this case, \(s \cdot_{op} t = ts\). Thereby, \(\cS\) is a right restriction semigroupoid under \(\cdot_{op}\) whenever it is a left restriction semigroupoid under the original operation, and vice versa.
    \end{obs}
    
    \begin{defi}
        Let $\cS$ be a semigroupoid. A \textit{subsemigroupoid} of $\cS$ is a triple $(\mathcal{T},\mathcal{T}^{(2)},\star)$ in which $\mathcal{T}$ is a non-empty subset of $\cS$, $\mathcal{T}^{(2)} = \mathcal{S}^{(2)} \cap (\mathcal{T} \times \mathcal{T})$ and $\star \colon \mathcal{T}^{(2)} \to \mathcal{T}$ is the restriction of the composition of $\cS$ to $\mathcal{T}$.
    \end{defi}

    In other words, a subsemigroupoid of a semigroupoid $\mathcal{S}$ is a subset $\mathcal{T}$ of $\cS$ that is closed under the composition of $\mathcal{S}$, and is regarded as a semigroupoid whose set of composable pairs consists of all \( (s, t) \in \mathcal{T} \times \mathcal{T} \) such that \( (s, t) \in \mathcal{S}^{(2)} \). Notice that if $(\cS,+)$ is a left restriction semigroupoid, then $\cS^+ = \{ s^+ \colon s \in \cS \}$ is a subsemigroupoid of $\cS$.\\
    
    For the rest of this section, assume that $\cS$ is a semigroupoid.\\
    
    Suppose that there is $E \subseteq E(\cS)$, $E \neq \emptyset$, such that if $e,f \in E$ and $(e,f) \in \cS^{(2)}$, then $(f,e) \in \cS^{(2)}$ and $ef = fe \in E$. That is, $E$ is an abelian subsemigroupoid of $E(\cS)$. Define in $E$ a partial order $\preceq_E$ given by:
    \begin{align} \label{defordememe}
        e \preceq_E f \text{ if and only if } (e,f) \in \cS^{(2)}  \text{ and } e = ef.
    \end{align} 
    The semigroupoids we consider are those that contain a subset \( E \) as described above. Such a subset \( E \) is called a \emph{distinct set} of \( \cS \). Note that \( \cS \) may have multiple distinct sets.
    
    \begin{lemma} \label{lema210}
        Assume that $\cS$ is left restriction. The following properties are valid.
        \begin{enumerate} \Not{p}
        \item $\cS^+ \subseteq E(\cS)$. \label{p1}
        \item $e=e^+$, for all $e \in \cS^+$. \label{p2}
        \item $\cS^+$ is a disjoint union of meet semilattices with respect to the partial order \( \preceq_{\cS^+} \). \label{p3}
        \end{enumerate}
    \end{lemma}
    \begin{proof}
        (p1): Notice initially that $(s^+,s^+) \in \cS^{(2)}$, since $(s^+,s) \in \cS^{(2)}$ and \eqref{261} holds. Now, by \eqref{lr3}, $s^+ = (s^+ s)^+  = s^+ s^+ $.
    
        (p2): Let $e \in \cS^+$. Then there exists $s \in \cS$ such that $e=s^+ $. Moreover, by \eqref{lr1}, $(e^+ ,e) \in \cS^{(2)}$ and $e^+ e = e$. Thus, applying the commutativity of $\cS^+ $, \eqref{lr3}, and \eqref{p1}, it follows that:
        \begin{align*}
            e=e^+ e=(s^+ )^+ s^+ =s^+ (s^+ )^+ = (s^+ s^+ )^+ =(s^+ )^+ =e^+ .
        \end{align*}

        (p3): If $e,f \in \cS^+$ are such that the meet $e \wedge f$ is defined, then $(e,f) \in \cS^{(2)}$ and $e \wedge f = ef$. Indeed, since \( (e,e \wedge f) \in \cS^{(2)} \) and \( (e \wedge f, f) \in \cS^{(2)} \), it follows from \eqref{s1} that \( (e(e \wedge f), f) \in \cS^{(2)} \). Moreover, by \eqref{lr2}, we obtain \( (e(e \wedge f))f = ((e \wedge f)e)f \). Applying \eqref{s2}, we conclude that \( (e,f) \in \cS^{(2)} \). Now, since \( ef \preceq_{\cS^+} e, f \), it follows that if \( j \in \cS^+ \subseteq E(\cS) \) satisfies \( j \preceq_{\cS^+} e,f \), then \( j = j^2 \preceq_{\cS^+} ef \).
    
        Now, define a relation $\sim$ in $\cS^+$ by $e \sim f$ if and only if $(e,f) \in \cS^{(2)}$. We will prove that \( \sim \) is an equivalence relation. Indeed, reflexivity follows from \eqref{p1}, and symmetry follows from \eqref{lr2}. For transitivity, assume that \( e \sim f \) and \( f \sim g \). Then \( (e,f) \in \cS^{(2)} \) and \( (f,g) \in \cS^{(2)} \), so by \eqref{s1}, we obtain \( (ef,g) \in \cS^{(2)} \). Moreover, by \eqref{lr2}, we have \( (ef)g = (fe)g \), and applying \eqref{s2}, we conclude that \( (e,g) \in \cS^{(2)} \), which means that \( e \sim g \).
    
        Since \( e \sim f \) implies that \( e \wedge f \) is defined, it follows that each \(\sim\)-class forms a meet semilattice in \( \cS^+ \). Moreover, since \(\sim\) is an equivalence relation, \( \cS^+ \) is the disjoint union of its \(\sim\)-classes.
    \end{proof}
    
    \begin{lemma} \label{lemma211}
        Assume that $\cS$ is left restriction. If $s \in S$ and $e \in \cS^+$ are such that $(e,s) \in \cS^{(2)}$ and $es = s$, then $s^+  \preceq_{\cS^+} e$. In other words, $s^+ $ is the minimum left identity for $s$ with respect to $\preceq_{\cS^+}$. 
    \end{lemma}
    \begin{proof}
        According to \eqref{p2}, we have \( e^+  = e \). Then, applying \eqref{lr3},
        \[
        s^+  = (es)^+  = (e^+ s)^+  = e^+ s^+  = es^+ .
        \]
        By definition, this implies that \( s^+  \preceq_{\cS^+} e \).
    \end{proof}
    
    \begin{lemma} \label{lema212}
       Assume that \(\cS\) is left restriction. Then, for all \((s,t) \in \cS^{(2)}\), we have \((st)^+  = (st^+ )^+ \) and \((st)^+ = (st)^+ s^+\).
    \end{lemma}
    \begin{proof}
        Consider \((s,t) \in \cS^{(2)}\). Then,  
            \[
            (st^+ )^+ (st) = (st^+ )^+ s(t^+ t) = (st^+ )^+ (st^+ )t = st^+ t = st,
            \]
        so by Lemma \ref{lemma211}, we obtain \((st)^+  \preceq_E (st^+ )^+ \). For the reverse inequality, applying \eqref{lr4}, we have 
            \[
            (st)^+ (st^+ ) = ((st)^+ s)t^+  = st^+ t^+  = st^+ ,
            \]
        which again, by Lemma \ref{lemma211}, gives \((st^+ )^+  \preceq_E (st)^+ \). On the other hand, applying \eqref{lr3}, \eqref{lr4} and the previous identity, we obtain
            $$ (st)^+ s^+ = ((st)^+ s)^+ = (st^+)^+ = (st)^+. $$
    \end{proof}
    
    \begin{prop} \label{lrcategorical}
        Every left restriction semigroupoid $\cS$ is categorical.
    \end{prop}
    \begin{proof}
      Let \( s \in \cS \). By \eqref{261}, we have \( \cS_s = \cS_{s^+ } \). Now, consider the relation \(\sim\) on \( \cS^+ \) defined by \( e \sim f \) if and only if \( (e, f) \in \cS^{(2)} \). In the proof of \eqref{p3}, we established that \(\sim\) is an equivalence relation. We denote by \( \tilde{e} \) the equivalence class of {\(e \in S^+\)} under \(\sim\).  
    
        \noindent\begin{enumerate}[left=0pt]\Not{Claim}
            \item $\cS_{s^+ } = \cS_e$, for all $e \in \tilde{s^+ }$.
            
            Indeed, let \( e \in \tilde{s^+ } \) and \( u \in \cS_{s^+ } \). Then, we have \( (s^+ , e), (u, s^+ ) \in \cS^{(2)} \). By \eqref{s3}, it follows that \( (u, s^+  e) \in \cS^{(2)} \), and by \eqref{lr2}, we obtain \( u(s^+  e) = u(e s^+ ) \). Now, since \( (e, s^+ ) \in \cS^{(2)} \) and \( (u, e s^+ ) \in \cS^{(2)} \), applying \eqref{s3} again gives \( (u, e) \in \cS^{(2)} \), which means that \( u \in \cS_e \). The reverse inclusion follows by symmetry.  \label{claim1}
            
            \item If \( e, f \in \cS^+ \) are such that \( e \not\sim f \), then \( \cS_e \cap \cS_f = \emptyset \). 
            
            Indeed, suppose that $e \not\sim f$. If there exists \( u \in \cS_e \cap \cS_f \), then \( (u, e) \in \cS^{(2)} \) and \( (u, f) \in \cS^{(2)} \). {Since $S_u = S_{u^+}$, we have that $(u^+,e) \in \cS^{(2)}$ and $(u^+,f) \in \cS^{(2)}$. But then $u^+ \in \tilde{e} \cap \tilde{f}$, contradicting $e \not\sim f$.} Therefore, \( \cS_e \cap \cS_f = \emptyset \).  \label{claim2}
        \end{enumerate}
        Let \( s, t \in \cS \). We must show that \( \cS_s \) and \( \cS_t \) are either equal or disjoint. Since \( \cS_s = \cS_{s^+ } \) and \( \cS_t = \cS_{t^+ } \), this follows from the fact that either \( s^+  \sim t^+  \) or \( s^+  \not\sim t^+  \), concluding the proof.  
    \end{proof}
    
    Therefore, we conclude by Theorem \ref{cord24} that every left restriction semigroupoid is a graphed semigroupoid. For the remaining of this subsection, assume that $\cS$ is a graphed semigroupoid with domain and range functions $D,R \colon \cS_1 \to \cS_0$.
    
    Now, we will give another equivalent definition of left restriction semigroupoids that involves semigroupoid congruences.

    \begin{defi}
        {A \textit{congruence} in a semigroupoid $\cS$ is an equivalence relation $\rho$ such that, for $s,t,u,v \in \cS$, if $s \rho t$, $u \rho v$ and $(s,u), (t,v) \in \cS^{(2)}$, then $(su) \rho (tv)$.}
    \end{defi}
    
    Our next characterization requires a more general notion of congruence, namely, left and right congruences.

    \begin{defi}
        {A \textit{left congruence} in a semigroupoid $\cS$ is an equivalence relation $\rho$ such that, for $s,t,r \in \cS$, if $s \rho t$ and $(r,s), (r,t) \in \cS^{(2)}$, then $(rs) \rho (rt)$. A \textit{right congruence} in $\cS$ is defined dually.}
    \end{defi}
    
    Let $E$ be a distinct set of $\cS$ and, {for each element $s \in \cS$, consider the set
        $$ R_E(s) = \{ e \in E \colon (s,e) \in \cS^{(2)}, se = s \}. $$
    That is, $R_E(s)$ is the set of right identities of $s$ in $E$. Then,} we define the relation:
    \begin{gather}\label{rel_right}
        {s \tilde{\mathcal{R}}_E t \iff R_E(s) = R_E(t).}
    \end{gather}
    Then $\tilde{\mathcal{R}}_E$ is an equivalence relation. Similarly, {let $L_E(s) = \{ e \in E \colon (e,s) \in \cS^{(2)}, es = s\}$ be the set of left identities of $s$ in $E$ and define the relation}:
    \begin{gather}\label{rel_left}
        {s \tilde{\mathcal{L}}_E t \iff L_E(s) = L_E(t).}
    \end{gather} 
    We have that $\tilde{\mathcal{L}}_E$ is an equivalence relation.

    \begin{prop}
        {Let $\cS$ be left restriction. Then $\cS^+$ is a distinct set, $\tilde{\mathcal{L}}_{\cS^+}$ is a left congruence and every $s \in \cS$ is $\tilde{\mathcal{L}}_{\cS^+}$ related with a unique idempotent in $\cS^+$, namely $s^+$. In particular, we have $s \tilde{\mathcal{L}}_{\cS^+} t$ if and only if $s^+ = t^+$.}

        \begin{proof}
            {We have that $\cS^+$ is a distinct set by Lemma \ref{lema210} and \eqref{lr2}. It follows from \eqref{claim1} that $s \tilde{\mathcal{L}}_{\cS^+} s^+$. Since $\tilde{\mathcal{L}}_{\cS^+}$ is an equivalence relation and $s^+ \in L_{\cS^+}(s^+)$, we obtain that:
                $$ s \tilde{\mathcal{L}}_{\cS^+} t \iff s^+ \tilde{\mathcal{L}}_{\cS^+} t^+ \iff s^+ = t^+s^+ \overset{\eqref{lr2}}{=} s^+t^+ = t^+. $$
            In particular, if $e \in \cS^+$ is such that $s \tilde{\mathcal{L}}_{\cS^+} e$, then $s^+ = e^+ = e$ by Lemma \ref{lema210}. Therefore, $s^+$ is the unique idempotent in $\cS^+$ related to $s$. Suppose that $s,t,r \in \cS$ are such that $s \tilde{\mathcal{L}}_{\cS^+} t$ and $(r,s), (r,t) \in \cS^{(2)}$. Then:
                $$ (rs)^+ \overset{\eqref{lema212}}{=} (rs^+)^+ = (rt^+)^+ \overset{\eqref{lema212}}{=} (rt)^+. $$
            From the previous arguments, this shows that $(rs) \tilde{\mathcal{L}}_{\cS^+} (rt)$. Hence, $\tilde{\mathcal{L}}_{\cS^+}$ is a left congruence, concluding the proof.}
        \end{proof}
    \end{prop}

    \begin{prop}
        {Let $E$ be a distinct set of $S$ such that:
        \begin{enumerate} \Not{R}
            \item every $s \in \cS$ is $\tilde{\mathcal{L}}_E$ related with a unique idempotent in $E$. Denote it by $s^+$. \label{r1}
            \item $\tilde{\mathcal{L}}_E$ is a left congruence. \label{r2}
            \item for all $s \in \cS$ and all $e \in E$ such that $(s,e) \in \cS^{(2)}$, we have that $((se)^+,s) \in \cS^{(2)}$ and $se = (se)^+s$. \label{r3}
        \end{enumerate}
        Then $\cS$ is left restriction with $\cS^+ = E$.}

        \begin{proof}
            {We need to verify condition \eqref{lr1} - \eqref{lr4}. We already have that \eqref{lr2} holds by definition of distinct set. Condition \eqref{lr1} follows the definition of $\tilde{\mathcal{L}}_E$ and $s \tilde{\mathcal{L}}_E s^+$, that is, $(s^+,s) \in \cS^{(2)}$ and $s^+s = s$. To verify \eqref{lr3} we note that $(t^+,t) \in \cS^{(2)}$ and $t^+t = t$ implies $(s,t) \in \cS^{(2)}$ if and only if $(s,t^+) \in \cS^{(2)}$. Therefore, since $\tilde{\mathcal{L}}_E$ is a left congruence and $t \tilde{\mathcal{L}}_E t^+$, it follows that if $(s,t) \in \cS^{(2)}$, then $(st) \tilde{\mathcal{L}}_E (st^+)$. That is, $(st)^+ = (st^+)^+$. Using \eqref{lr1} and \eqref{lr2}, we obtain:
                $$ (s^+t^+)(s^+t) = s^+t^+s^+t = s^+s^+t^+t = s^+t^+t = s^+t. $$
            Furthermore, since $(s^+t)^+ = (s^+t^+)^+ = s^+t^+$, it follows that:
                $$ (s^+t)^+ s^+t^+ = (s^+t^+)^+ s^+t^+ = s^+t^+. $$
            From this, we conclude that \(s^+ t \tilde{\mathcal{L}}_E s^+ t^+ \). Applying \eqref{r1}, we obtain \((s^+ t)^+  = (s^+ t^+ )^+  = s^+ t^+ \). Finally, if $(s,t) \in \cS^{(2)}$, then:
                $$ st^+ \overset{(3)}{=} (st^+)^+ s \overset{(2)}{=} (st)^+ s. $$
            Thus, condition \eqref{lr4} holds, concluding that $\cS$ is left restriction. From (1) we also obtain that $e^+ = e$, for all $e \in E$. Hence, it must be $\cS^+ = E$.}
        \end{proof}
    \end{prop}

The definition of a left restriction semigroupoid extends the concepts of both left restriction semigroups \cite[Definition 2.1]{gould2009partial} and left restriction categories \cite[Definition 2.1.1]{cockett2002restriction}. It is worth noting that we have chosen the term ``left restriction" rather than ``weakly left \( E \)-ample" because, as shown in \cite[Theorem 4.13]{hollings2009PP}, weakly left \( E \)-ample semigroups are precisely left restriction semigroups. 

    \section{A representation theorem} The Wagner-Preston Theorem allows us to identify any inverse semigroup with a subsemigroup of an inverse monoid of partial bijections of a set. In this section, we prove that every restriction semigroupoid can be viewed as a restriction subsemigroupoid of a restriction category consisting of partial maps between the fibers of a function. Throughout this section, $\cS$ and $\mathcal{T}$ will denote semigroupoids.\\

    We begin by recalling the definition of the semigroup of partial maps of a set. Let $X$ be a set. A \textit{partial map} of $X$ is a (surjective) function $f \colon {_fX} \to X_f$, where ${_fX}, X_f \subseteq X$. Let $\mathcal{PT}(X)$ be the family of all partial maps of $X$. Then $\mathcal{PT}(X)$ becomes a monoid with composition
        $$ f \star g \colon g^{-1}(X_g \cap {_fX}) \mapsto f(X_g \cap {_fX}), \ (f \star g)(x) = f(g(x)). $$
    Furthermore, $\mathcal{PT}(X)$ has a right restriction structure, given by $f^\ast = id_{_fX}$.  The triple $(\mathcal{PT}(X),\star,\ast)$ is the right restriction monoid of partial maps of $X$. It is easy to see that $\mathcal{PT}(X)^{op} := (\mathcal{PT}(X),\bullet,+)$, where
    \begin{align} \label{PTXop}
        f \bullet g = g \star f \quad\text{and}\quad f^+ = f^\ast,
    \end{align}
    is a left restriction monoid.

    \begin{defi}
        Let $\pi \colon X \to Y$ be a surjective function. A \textit{partial map between fibers of $\pi$} is a (surjective) function $f \colon {_fX} \to X_f$ with ${_fX} \subseteq \pi^{-1}(y)$ and $X_f \subseteq \pi^{-1}(x)$.
    \end{defi}

    For the rest of this section, $\pi$ will denote a surjective function $\pi \colon X \to Y$. Notice that the family of partial maps between fibers of $\pi$ is a subset of $\mathcal{PT}(X)$.

    \begin{prop}
        {Consider the following data:}
        \begin{itemize}
            \item {$\mathcal{PT}(\pi)_0 = Y$;}
            \item {$\mathcal{PT}(\pi)_1 = \{ (y,f,x) \colon f \text{ is a partial map } \pi^{-1}(x) \to \pi^{-1}(y) \}$;}
            \item {For each $(y,f,x) \in \mathcal{PT}(\pi)_1$, define $R(y,f,x) = y$ and $D(y,f,x) = x$;}
            \item {For each pair $(z,g,y), (y,f,x) \in \mathcal{PT}(\pi)_1$, define $(z,g,y) \circ (y,f,x) = (z,g \star f, x)$, where $g \star f$ denotes the composition in the monoid of partial maps $\mathcal{PT}(X)$;}
            \item {For each $y \in Y$, define $1_y = (y,id_{\pi^{-1}(y)},y)$.}
        \end{itemize}
        {Then $\mathcal{PT}(\pi) := (\mathcal{PT}(\pi)_0, \mathcal{PT}(\pi)_1, D, R, \circ)$ is a category whose objects are given by $\mathcal{PT}(\pi)_0$, the morphisms are given by $\mathcal{PT}(\pi)_1$, the range and the domain are given, respectively, by $R$ and $D$, the composition between two morphisms is defined as above, and, for each $y \in \mathcal{PT}(\pi)_0$, the identity is given by $1_y$.}

        \begin{proof}
            {
            We first verify that the composition $\circ$ is well defined. Let $(z,g,y), (y,f,x) \in \mathcal{PT}(\pi)_1$. Then
                $$ dom(g \star f) = f^{-1}(X_f \cap {_gX}) \subseteq {_fX} \subseteq \pi^{-1}(x), $$
            and
                $$ im(g \star f) = g(X_f \cap {_gX}) \subseteq X_g \subseteq \pi^{-1}(z). $$
            That is, $g \star f$ is a partial function $\pi^{-1}(x) \to \pi^{-1}(z)$ and, hence, $(z,g \star f,x) \in \mathcal{PT}(\pi)_1$. We also have that the composition is compatible with the range and domain functions, since
                $$ R((z,g,y) \circ (y,f,x)) = R(z,g \star f,x) = z = R(z,g,y), $$
            and
                $$ D((z,g,y) \circ (y,f,x)) = D(z,g \star f,x) = x = D(y,f,x). $$
            Now we check that the composition is associative. Let $(w,h,z) \in \mathcal{PT}(\pi)_1$. Then
            \begin{align*}
                (w,h,z) \circ [(z,g,y) \circ (y,f,x)] &= (w,h \star (g \star f),x) \\
                &= (w,(h \star g) \star f,x) \\
                &= [(w,h,z) \circ (z,g,y)] \circ (y,f,x),
            \end{align*}
            where the second equality follows from the associativity of the monoid $(\mathcal{PT}(X),\star)$. Now we prove that, for each $y \in Y$, $1_y = (y,id_{\pi^{-1}(y)},y)$ is, in fact, an identity. Let $(y,f,x) \in \mathcal{PT}(\pi)_1$. Then $f \star id_{\pi^{-1}(x)} = f$ since $dom(f) \subseteq \pi^{-1}(x)$, and $id_{\pi^{-1}(y)} \star f = f$ since $im(f) \subseteq \pi^{-1}(y)$. Consequently, we obtain
                $$ (y,f,x) = (y, id_{\pi^{-1}(y)} \star f, x) = (y,id_{\pi^{-1}(y)},y) \circ (y,f,x), $$
            and
                $$ (y,f,x) = (y, f \star id_{\pi^{-1}(x)}, x) = (y,f,x) \circ (x, id_{\pi^{-1}(x)},x). $$
            This concludes that $\mathcal{PT}(\pi) = (\mathcal{PT}(\pi)_0, \mathcal{PT}(\pi)_1, D,R, \circ)$ is a category.}
        \end{proof}
    \end{prop}

    \begin{defi}
        The category {$\mathcal{PT}(\pi)$} will be referred to as the \textit{category of {partial} maps between fibers of $\pi$}.
    \end{defi}

    The next example allows us to obtain the monoid of partial maps $\mathcal{PT}(X)$ as a particular case of the category of partial maps between fibers of $\pi$.

    \begin{exe} \label{exe:PTpi-semigroup}
        Let $X$ be a set and $\pi \colon X \to \{1\}$ be the trivial function. Then {$\mathcal{PT}(\pi)$} is isomorphic to {$\mathcal{PT}(X)$} as monoid. Indeed, since $\pi^{-1}(1) = X$, we have 
            $$ \mathcal{PT}(\pi)_1 = \{ (1,f,1) \colon f \text{ is a partial map } X \to X \}, $$
        and the composition in {$\mathcal{PT}(\pi)$} is given by {$(1,g,1) \circ (1,f,1) = (1,g \star f,1)$}. Therefore, $(1,f,1) \mapsto f$ defines an isomorphism {$\mathcal{PT}(\pi) \to \mathcal{PT}(X)$}.
    \end{exe}

    The {right} restriction structure on {$\mathcal{PT}(X)$} induces a {right} restriction structure on {$\mathcal{PT}(\pi)$}.

    \begin{prop}
        {The category $\mathcal{PT}(\pi)$ of partial maps between fibers of $\pi$ has a right restriction structure, given by $(y,f,x)^\ast = (x,id_{_fX},x)$, for every $(y,f,x) \in \mathcal{PT}(\pi)_1$.}

        \begin{proof}
            {Since $\mathcal{PT}(\pi)$ is a category, it is enough to verify the following conditions:
            \begin{itemize}
                \item[\eqref{rr1}] $(y,f,x) \circ (x,id_{_fX},x) = (y,f,x)$;
                \item[\eqref{rr2}] $(y,id_B,y) \circ (y,id_A,y) = (y,id_B,y) \circ (y,id_A,y)$, for every $A,B \subseteq \pi^{-1}(y)$;
                \item[\eqref{rr3}] $((y,f,x) \circ (x,id_A,x))^\ast = (x,id_{_fX},x) \circ (x,id_A,x)$ for every $A \subseteq \pi^{-1}(y)$;
                \item[\eqref{rr4}] $(y,id_{_gX},y) \circ (y,f,x) = (y,f,x) \circ ((z,g,y) \circ (y,f,x))^\ast$.
            \end{itemize}
            Conditions \eqref{rr1} and \eqref{rr2} follows from $f \star id_{_fX} = f$ and $id_A \star id_B = id_{A \cap B} = id_B \star id_A$, respectively. Since $dom(f \star 1_A) = {_fX} \cap A$, we have
            \begin{align*}
                (y,f,x) \circ (x,id_A,x))^\ast &= (y,f \star id_A,x)^\ast = (x,id_{{_fX} \cap A},x) \\
                &= (x,id_{_fX} \star id_A,x) \\
                &= (x, id_{_fX},x) \circ (x,id_A,x).
            \end{align*}
            Therefore, condition \eqref{rr3} is satisfied. Let $f,g \in \mathcal{PT}(X)$ and $A = dom(g \star f) = f^{-1}(X_f \cap {_gX})$. Since $(\mathcal{PT}(X),\star,\ast)$ is a right restriction monoid, it follows from \eqref{rr4} that
                $$ id_{_gX} \star f = (g \star f)^\ast \star f = id_A \star f. $$
            Thus,
            \begin{align*}
                (y,id_{_gX},y) \circ (y,f,x) &= (y,id_{_gX} \star f,x) \\
                &= (y,id_A \star f,x) \\
                &= (y,id_A,y) \circ (y,f,x) \\
                &= (z,g \star f,x)^\ast \circ (y,f,x) \\
                &= ((z,g,y) \circ (y,f,x))^\ast \circ (y,f,x).
            \end{align*}
            This proves \eqref{rr4}. Hence, $(\mathcal{PT}(\pi),\ast)$ is a right restriction category.}
        \end{proof}
    \end{prop}

    \begin{obs}
        In Section 3 of \cite{cockett2002restriction}, Cockett and Lack introduced the category of partial maps $Par(\mathcal{C},\mathcal{M})$, which is a restriction category. Observe that $\mathcal{PT}(\pi)$ can be identified with the full subcategory of $Par(Set,\mathcal{M})$ whose objects are precisely the sets $\pi^{-1}(y)$, for $y \in Y$, where $Set$ denotes the category of sets and $\mathcal{M}$ is the class of all injective functions. The restriction structure on $\mathcal{PT}(\pi)$ is inherited from that of $Par(Set,\mathcal{M})$.

        Although $\mathcal{PT}(\pi)$ is already known as a restriction subcategory of $Par(Set,\mathcal{M})$, we choose to present its construction with complete proofs, as this elementary approach differs from that of Cockett and Lack.
    \end{obs}

    Before we present the representation theorem for left restriction semigroupoids, we need to stablish what we mean with \emph{``seen as a restriction subsemigroupoid''}.

    \begin{defi}
        A function $f \colon \cS \to \mathcal{T}$ is called:
        \begin{itemize}
            \item a \textit{morphism} if $(s,t) \in \cS^{(2)}$ implies $(f(s),f(t)) \in \mathcal{T}^{(2)}$ and $f(st) = f(s)f(t)$.
            \item a \textit{rigid morphism} if it is a morphism and $(f(s),f(t)) \in \mathcal{T}^{(2)}$ implies $(s,t) \in \cS^{(2)}$.
            \item an \textit{embedding} if it is an injective rigid morphism.
            \item an \textit{isomorphism} if it is a bijective rigid morphism.
        \end{itemize}
        If, in addition, $(\cS,+)$ and $(\mathcal{T},+)$ are restriction semigroupoids, then $f$ is called a \emph{restriction morphism} (respectively, \emph{restriction rigid morphism}, \emph{restriction embedding}, or \emph{restriction isomorphism}) if it is a morphism (respectively, rigid morphism, embedding, or isomorphism) and satisfies $f(s^+) = f(s)^+$, for all $s \in \cS$.
    \end{defi}

    \begin{defi}
        We say that $\cS$ and $\mathcal{T}$ are \textit{isomorphic} if and only if there exist {an isomorphism $f \colon \cS \to \mathcal{T}$}. If $\cS$ and $\mathcal{T}$ are left restriction semigroupoids, we say that $\cS$ and $\mathcal{T}$ are isomorphic \textit{as left restriction semigroupoids} if they are isomorphic with $f \colon \cS \to \mathcal{T}$ {a restriction isomorphism}.
    \end{defi}

    {Note that, if $\mathcal{C} = (\mathcal{C}_0,\mathcal{C}_1,D,R,\circ)$ is a category, then $\mathcal{C}^{op} = (\mathcal{C}_0,\mathcal{C}_1,R,D,\bullet)$ is also a category, where composition is defined as $f \bullet g := g \circ f$. Furthermore, $(\mathcal{C},\ast)$ is right restriction if and only if $(\mathcal{C}^{op},+)$ is left restriction, where $f^+ = f^\ast$. In the following we prove that, for every left restriction semigroupoids $\cS$, there exists a restriction embedding $\cS \to \mathcal{PT}(\pi)^{op}$, for some function $\pi$. Thus, every left restriction semigroupoid is isomorphic to, or ``can be seen as'', a restriction subsemigroupoid of some left restriction category $\mathcal{PT}(\pi)^{op}$.}

    \begin{theorem} \label{thm:Wagner-Preston-semigroupoid}
        Let $\cS$ be a left restriction semigroupoid. Then there is a restriction embedding $\cS \to \mathcal{PT}^{op}(\pi)$, for some function $\pi \colon X \to Y$.

        \begin{proof}
            Since every left restriction semigroupoid is categorical by Proposition \ref{lrcategorical}, we can find range and domain functions $R,D \colon \cS \to \cS^{(0)}$ compatible with the composition of $\cS$. We notice that $(R,D,\cS^{(0)})$ may not be uniquely determined. For each $s \in \cS$, define sets
                $$ {_sX} = \{ t \in \cS \colon (t,s^+) \in \cS^{(2)}, ts^+ = t \}, $$
            and $X_s = \{ ts \colon t \in {_sX} \}$, and a map $\alpha_s \colon {_sX} \to X_s$, given by $\alpha_s(t) = ts$, for every $t \in {_sX}$. The map $\alpha_s$ is well defined since, from \eqref{lr1}, we obtain that $(t,s^+) \in \cS^{(2)}$ if and only if $(t,s) \in \cS^{(2)}$. Furthermore, we have that
                $$ t \in {_sX} \implies (t,s) \in \cS^{(2)} \iff D(t) = R(s) \iff t \in D^{-1}(R(s)), $$
            and if $(t,s) \in \cS^{(2)}$, then $D(ts) = D(s)$. Therefore, $\alpha_s(t) \in D^{-1}(D(s))$, that is, $\alpha_s$ is a partial map $D^{-1}(R(s)) \to D^{-1}(D(s))$. We can thus define a function $\alpha \colon \cS \to \mathcal{PT}(D)^{op}$ by $s \mapsto (D(s),\alpha_s,R(s))$. In what follows, we show that $\alpha$ is a restriction embedding.

            To prove that $\alpha$ is a rigid morphism, notice that, {in the opposite category $\mathcal{PT}(\pi)^{op}$, we have $D(y,f,x) = y$ and $R(y,f,x) = x$. Hence,}
            \begin{align*}
                (s,t) \in \cS^{(2)} &\iff D(s) = R(t) \\
                &\iff D(D(s),\alpha_s,R(s)) = R(D(t),\alpha_t,R(t)) \\
                &\iff (\alpha(s),\alpha(t)) \in (\mathcal{PT}(D)^{op})^{(2)}.
            \end{align*}
            Let $(s,t) \in \cS^{(2)}$. Then
                $$ \alpha(s) \bullet \alpha(t) = \alpha(t) \circ \alpha(s) = (D(t),\alpha_t \star \alpha_s,R(s)) \quad\text{and}\quad \alpha(st) = (D(st),\alpha_{st},R(st)). $$
            Since the range and domain functions are compatible with the composition in $\cS$, we have $D(st) = D(t)$ and $R(st) = R(s)$. On the other hand, note that
                $$ dom(\alpha_{st}) = {_{st}X} = \{ u \in \cS \colon (u,(st)^+) \in \cS^{(2)}, u(st)^+ = u \}, $$
            and
                $$ dom(\alpha_t \star \alpha_s) = \alpha_s^{-1}(X_s \cap {_tX}) = \{ u \in \cS \colon (u,s^+), (us,t^+) \in \cS^{(2)}, us^+ = u, ust^+ = us \}. $$
            Let $u \in dom(\alpha_{st})$. From Lemma \ref{lema212} and \eqref{lr2} we obtain $(st)^+ = (st)^+ s^+ = s^+ (st)^+$. Therefore, $(u,(st)^+) \in \cS^{(2)}$ implies $(u,s^+) \in \cS^{(2)}$, which is equivalent to $(u,s) \in \cS^{(2)}$. Since $(s,t) \in \cS^{(2)}$, it follows that $(us,t) \in \cS^{(2)}$, or equivalently, $(us,t^+) \in \cS^{(2)}$. Thereby, we have
            \begin{align*}
                us^+ &= u(st)^+ s^+ & (u \in {_{st}X}) \\
                &= u(st)^+ = u, & \eqref{lema212}
            \end{align*}
            and
            \begin{align*}
                ust^+ &= u(st)^+ s & \eqref{lr4} \\
                &= us. & (u \in {_{st}X})
            \end{align*}
            This proves $dom(\alpha_{st}) \subseteq dom(\alpha_t \star \alpha_s)$. Let {$u \in dom(\alpha_t \star \alpha_s)$}. Since $(u,s^+),(us,t^+) \in \cS^{(2)}$, it follows that $(u,st^+) = (u,(st)^+ s) \in \cS^{(2)}$ and, hence, $(u,(st)^+) \in \cS^{(2)}$. Therefore,
            \begin{align*}
                u(st)^+ &= (ust)^+ u & \eqref{lr4} \\
                &= (ust^+)^+ u & \eqref{lema212} \\
                &= (us)^+ u & (us \in {_tX}) \\
                &= (us^+)^+ u & \eqref{lema212} \\
                &= u^+ u & (u \in {_sX}) \\
                &= u. & \eqref{lr1}
            \end{align*}
            This shows that $dom(\alpha_{st}) = dom(\alpha_t \star \alpha_s)$. Let $u \in dom(\alpha_{st})$. Then
                $$ (\alpha_t \star \alpha_s)(u) = (us)t = u(st) = \alpha_{st}(u). $$
            This concludes that $\alpha(s) \bullet \alpha(t) = \alpha(st)$ and, hence, $\alpha$ is a rigid morphism.
            
            To see that $\alpha$ is a restriction morphism, consider $s \in \cS$ and observe that
                $$ {_{s^+}X} = \{t \in \cS \colon (t,s^+) \in \cS^{(2)}, ts^+ = t\} = {_sX}. $$
            In particular, we obtain that $\alpha_{s^+} = id_{_sX}$. Furthermore, from $(s^+,s),(s^+,s^+) \in \cS^{(2)}$, we have that $R(s^+) = D(s^+) = R(s)$ and, hence,
                $$ \alpha(s^+) = (D(s^+),\alpha_{s^+},R(s^+)) = (R(s),id_{_sX},R(s)) = \alpha(s)^+. $$
            
            Next we verify that $\alpha$ is injective. Suppose that $\alpha(s) = \alpha(t)$. Since $(s^+,s) \in \cS^{(2)}$ and $s^+ s^+ = s^+$, we have $s^+ \in {_sX}$. Therefore, from \eqref{lr1} and $\alpha_s = \alpha_t$, we obtain
                $$ s = s^+ s = \alpha_s(s^+) = \alpha_t(s^+) = s^+ t. $$
            Analogously, we prove that $t^+ \in {_tX}$ and, hence, $t = t^+ s$. By \eqref{lr1} and \eqref{lr2}, it follows that
                $$ s = s^+ t = s^+ t^+ t = t^+ s^+ t = t^+ s = t. $$
            This concludes that $\alpha \colon \cS \to \mathcal{PT}^{op}(D)$ is an injective rigid morphism that preserves the restriction structure, hence, a restriction embedding.
        \end{proof}
    \end{theorem}

    Now, we present \cite[Theorem 2.32]{cordeiro2023etale} as a corollary of Theorem \ref{thm:Wagner-Preston-semigroupoid}. We denote by $\mathcal{J}(\pi)$ the subcategory of $\mathcal{PT}(\pi)$ whose morphisms are the triples $(y,f,x)$, where $f \colon {_fX} \to X_f$ is a bijection. In this case, $\mathcal{J}(\pi)$ forms an inverse category with $(y,f,x)^{-1} = (x,f^{-1},y)$.

    \begin{corollary} \cite[Theorem 2.32]{cordeiro2023etale} \label{coro:Wagner-Preston-ISGD}
        Let $\cS$ be an inverse semigroupoid. Then there is an inverse preserving embedding $\cS \to \mathcal{J}(\pi)$, for some function $\pi \colon X \to Y$.

        \begin{proof}
            Every inverse semigroupoid is a {right} restriction semigroupoid with $s^\ast = s^{-1}s$. Therefore, $\cS^{op}$ is left restriction with $s^+ = ss^{-1}$. Now it follows from Theorem \ref{thm:Wagner-Preston-semigroupoid} that there is an embedding $\alpha \colon \cS^{op} \to \mathcal{PT}(D)^{op}$, or equivalently, $\alpha \colon \cS \to \mathcal{PT}(D)$. We will prove that $im(\alpha) \subseteq \mathcal{J}(D)_1$, that is, $\alpha_s$ is a bijection for every $s \in \cS$. {Since $\alpha_s$ is chosen to be surjective, we only need to prove that $\alpha_s$ is injective. In fact, suppose that $u,v \in {_sX}$ are such that $\alpha_s(u) = \alpha_s(v)$. Then
                $$ u = uss^{-1} = \alpha_s(u)s^{-1} = \alpha_s(v)s^{-1} = vss^{-1} = v, $$
            where the first and last equalities follows from $u,v \in {_sX}$. Hence, $\alpha_s$ is bijective.}

            We notice that $\alpha \colon \cS \to \mathcal{J}(\pi)$ preserves the inverse operation. Indeed, for every $s \in \cS$ we have
                $$ \alpha(s)\alpha(s^{-1})\alpha(s) = \alpha(ss^{-1}s) = \alpha(s), $$
            and analogously $\alpha(s^{-1})\alpha(s)\alpha(s^{-1}) = \alpha(s^{-1})$. Since $\mathcal{J}(\pi)$ is an inverse category, the element satisfying these identities is unique. Therefore, it must be $\alpha(s^{-1}) = \alpha(s)^{-1}$.
        \end{proof}
    \end{corollary}

    Let $X$ be a set. The monoid of partial bijections of $X$ is the subsemigroup $\mathcal{J}(X)$ of $\mathcal{PT}(X)$ whose elements are bijections $f \colon {_fX} \to X_f$. Based on the previous discussion, we can identify $\mathcal{J}(X) \simeq \mathcal{J}(\pi)$ where $\pi \colon X \to \{1\}$ is the trivial function. By combining Corollary \ref{coro:Wagner-Preston-ISGD} with Example \ref{exe:PTpi-semigroup}, we recover the Wagner–Preston Theorem for inverse semigroups.

    \begin{corollary} \cite[Theorem 6]{wagner1952semigroups} \cite[Theorem 1]{preston1954semigroups}
        Let $S$ be an inverse semigroup. Then there is an inverse preserving embedding $S \to \mathcal{J}(X)$, for some set $X$.

        \begin{proof}
            Since $S$ is a semigroup, the range and domain functions are uniquely determined as $D,R \colon S \to \{1\}$. Since every inverse semigroup is, in particular, an inverse semigroupoid, it follows from Theorem \ref{coro:Wagner-Preston-ISGD} that there is a restriction embedding $S \to \mathcal{J}(D) \simeq \mathcal{J}(S)$.
        \end{proof}
    \end{corollary}

    \section{The Szendrei Expansion}
    
    The Szendrei expansion of a left restriction semigroup $S$, as defined in \cite[Definition 4.2]{gould2009partial}, is a left restriction semigroup $Sz(S)$ such that the premorphisms from $S$ {to left restriction semigroups} correspond bijectively to morphisms from $Sz(S)$ \cite[Theorem 5.2]{gould2009partial}. A natural question is whether this construction extends to the setting of restriction semigroupoids. The answer is affirmative, and we establish this correspondence in Theorem \ref{teo_princ}.
    
\subsection{The main construction}

    The (left restriction) monoid of partial maps $\mathcal{PT}(X)^{op}$ of a set $X$ has a natural order $\leq$, defined as
        $$ f \leq g \iff dom(f) \subseteq dom(g) \text{ and } f = g|_{dom(f)}. $$
    On the other hand, for arbitrary $f,g \in \mathcal{PT}(X)^{op}$ and the operation $\bullet$ as defined in \eqref{PTXop}, we have $dom(f^+ \bullet g) = dom(g \star id_{dom(f)}) = dom(f) \cap dom(g)$ and $(f^+ \bullet g)(x) = g(x)$, for every $x \in dom(f) \cap dom(g)$. Now, it is easy to see that
        $$ f^+ \bullet g = f \iff dom(f) \subseteq dom(g) \text{ and } g(x) = f(x), \ \forall x \in dom(f) \iff f \leq g. $$
        
    Inspired by this relation, we define a relation $\preceq$ in a left restriction semigroupoid $\cS$ by
    \begin{align} \label{nat_ord_gen}
        s \preceq t \iff (s^+,t) \in \cS^{(2)} \text{ and } s^+ t = s.
    \end{align}
   Notice that \( s \preceq t \) if and only if there exists \( e \in \cS^+  \) such that \( (e, t) \in \mathcal{S}^{(2)} \) and \( et = s \). Indeed, if \( s \preceq t \), then \( s^+  \in \cS^+ \), \( (s^+ , t) \in \mathcal{S}^{(2)} \), and \( s^+  t = s \). Conversely, if \( et = s \) for some \( e \in \cS^+ \), then, by Lemma \ref{lema212} and condition \eqref{lr3}, we have
 \[
s^+  = (et)^+  = (e t^+ )^+  = e t^+ ,
\]
    {therefore,
        $$ s = et = e t^+  t = s^+  t. $$}
    Furthermore, we notice that the relation $\preceq$ coincides with the partial order $\preceq_{\cS^+}$ on $\cS^+$, introduced in \eqref{defordememe}. In fact, for $e,f \in \cS^+$, we have $e \preceq f$ if and only if $(e,f) = (e^+,f) \in \cS^{(2)}$ and $ef = e^+ f = e$.
    
    In the case of inverse semigroupoids, $\cS^+ = E(\cS)$ and $\preceq$ is the natural partial order of $\cS$, as defined in \cite[\textsection 2.1]{liu2016free}. It is clear that if $s\in \cS$ and $e \in \cS^+$ are such that $(e,s) \in \cS^{(2)}$, then $es \preceq s$. The next result analyses the case in where $(s,e) \in \cS^{(2)}$.
    
    \begin{lemma} \label{prop221}
        If \( s \in \cS \) and \( e \in \cS^+ \) are such that \( (s,e) \in \cS^{(2)} \), then \( se \preceq s \).
    \end{lemma}
    \begin{proof}
        From \eqref{lr4} we have $se = (se)^+ s$. Since {$(se)^+ \in \cS^+$}, it follows that $(se)^+ s \prec s$.
    \end{proof}
    
    From this point forward, assume that \( \cS \) and \( \T \) are left restriction semigroupoids. We will use the notation \( \preceq \) to denote the partial order defined in \eqref{nat_ord_gen} for both \( \cS \) and \( \T \).
    
    \begin{defi} \label{defi:premorphism}
        We say that a map $\varphi : \cS \to \T$ is a \emph{premorphism} if the following conditions are satisfied:
        \begin{enumerate}
            \item[(a)] If $st$ is defined, then $\varphi(s)\varphi(t)$ and $\varphi(s)^+ \varphi(st)$ are defined and $\varphi(s)\varphi(t) = \varphi(s)^+ \varphi(st)$;
            
            \item[(b)] $\varphi(s)^+  \preceq \varphi(s^+ )$, {for all $s \in S$}.
        \end{enumerate}
    \end{defi}
    
    Next, we establish some properties related to premorphisms of left restriction semigroupoids.
    
    \begin{lemma}
        Let $\varphi : \cS \to \T$ be a premorphism. Then:
        \begin{enumerate}\Not{pm}
            \item If $e \in E(\cS)$, then $\varphi(e) \in E(\T)$.\label{pm1}
            
            \item If {$e \in \cS^+$}, then {$\varphi(e) \in \mathcal{T}^+$}.\label{pm2}
            
            \item If {$e,f \in \cS^+$} and $e \preceq f$, then $\varphi(e) \preceq \varphi(f)$.\label{pm3}
            
            \item If $u,v \in \cS$ are such that $u \preceq v$, then $\varphi(u) \preceq \varphi(v)$.\label{pm4}
            
            \item For all $(s,t) \in \cS^{(2)}$, $\varphi(s)\varphi(t) = \varphi(st^+ )\varphi(t).$\label{pm5}
        \end{enumerate}
    \end{lemma}
    \begin{proof}
        (\textit{pm1}): If $e \in E(\cS)$, then $ee$ is defined and $ee = e$. Now, it follows from Definition \ref{defi:premorphism}(a) and \eqref{lr1} that $\varphi(e)\varphi(e) = \varphi(e)^+ \varphi(ee) = \varphi(e)^+ \varphi(e) = \varphi(e)$. That is, $\varphi(e) \in E(\T)$.\\
        
        (\textit{pm2}): Assume that {$e \in \cS^+$}. Then, by Definition \ref{defi:premorphism}(b), we have $\varphi(e)^+ \preceq \varphi(e^+)$. That is, $\varphi(e)^+ = \varphi(e)^+ \varphi(e^+)$. Since $e \in E$, it follows from Lemma \ref{lema210} and \eqref{lr1} that $\varphi(e)^+ = \varphi(e)^+ \varphi(e) = \varphi(e)$. Hence, {$\varphi(e) \in \mathcal{T}^+$}.\\
    
        (\textit{pm3}): Consider {$e,f \in \cS^+$} such that $e \preceq f$. Then $ef$ is defined and $e = ef$. Applying Definition \ref{defi:premorphism}(a), \eqref{pm2} together with Lemma \ref{lema210}, and \eqref{pm1}, we obtain 
            $$ \varphi(e)\varphi(f) = \varphi(e)^+ \varphi(ef) = \varphi(e)\varphi(ef) = \varphi(e)\varphi(e) = \varphi(e). $$
        Therefore $\varphi(e) \preceq \varphi(f)$.\\
        
        (\textit{pm4}): Let $u,v \in \cS$ be such that $u \preceq v$. Then there is {$e \in \cS^+$} such that $ev$ is defined and $u = ev$. Applying Definition \ref{defi:premorphism}(a), \eqref{pm2} together with Lemma \ref{lema210}, and using that $u \preceq v$, it follows that
            $$ \varphi(e)\varphi(v) = \varphi(e)^+ \varphi(ev) = \varphi(e)\varphi(ev) = \varphi(e)\varphi(u). $$
        On the other hand, since $u = ev$, we can use Lemma \ref{lema212} to obtain $u^+ = (ev)^+ = (ev)^+ e^+$. That is, $u^+ \leq e^+ = e$. By Definition \ref{defi:premorphism}(b) and \eqref{pm3} we have $\varphi(u)^+ \leq \varphi(u^+) \leq \varphi(e)$. Therefore, multiplying the previous identity by $\varphi(u)^+$ on the left, we obtain
            $$ \varphi(u)^+ \varphi(v) = \varphi(u)^+ \varphi(e) \varphi(v) = \varphi(u)^+ \varphi(e) \varphi(u) = \varphi(u)^+ \varphi(u) = \varphi(u). $$
        Where, in the last identity, we use \eqref{lr1}. Therefore, $\varphi(u) \preceq \varphi(v)$.\\
        
        (\textit{pm5}): If $st$ is defined then $st^+$ is defined, and hence $(st^+)t$ is defined. Therefore, it follows from Definition \ref{defi:premorphism}(a) that $\varphi(s)\varphi(t)$ and $\varphi(st^+)\varphi(t)$ are defined. Furthermore, since the partial order $\preceq$ is compatible with composition, we have
        \begin{align*}
            \varphi(s)\varphi(t) &= \varphi(s) \varphi(t)^+ \varphi(t) & \eqref{lr1} \\
            &\preceq \varphi(s) \varphi(t^+) \varphi(t) & \ref{defi:premorphism}(b) \\
            &\preceq \varphi(st^+) \varphi(t). & \ref{defi:premorphism}(a)
        \end{align*}
        On the other hand we have $st^+ \preceq s$ by Lemma \ref{prop221}, $\varphi(st^+) \preceq \varphi(s)$ by \eqref{pm4} and therefore $\varphi(st^+)\varphi(t) \preceq \varphi(s)\varphi(t)$ from the fact that $\preceq$ is compatible with composition. Since $\preceq$ is a partial order, we conclude that $\varphi(s)\varphi(t) = \varphi(st^+)\varphi(t)$.
    \end{proof}
    
 We now begin the construction of the Szendrei expansion of a left restriction semigroupoid.
    
    \begin{defi}
        The Szendrei expansion $Sz(\cS)$ of $\cS$ is defined as the set
        \begin{align*}
            Sz(\cS) = \{ (A,a) \in \mathcal{P}^f(\cS) \times \cS : a, a^+  \in A \text{ and } {b^+ = a^+, \forall b \in A} \}
        \end{align*}
        {where $\mathcal{P}^f(\cS)$ is the set of all finite subsets of $\cS$,} equipped with the following partial binary operation:
        \begin{align*}
            (A,a)(B,b) = \begin{cases}
                ((ab)^+ A \cup aB, ab), \text{ if } \exists ab, \\
                \text{undefined, otherwise,}
            \end{cases}
        \end{align*}
    \end{defi}
    
    Observe that this operation is well defined. Indeed, {since $b^+ = c^+$, for all $c \in B$, and $ab$ is defined if and only if $ab^+$ is defined, it follows that $ac$ is defined, for all $c \in B$, and} in this case $(ac)^+  = (ac^+)^+ = (ab^+)^+ = (ab)^+$. On the other hand, it follows from Lemma \ref{lema212} that  $(ab)^+ = (ab)^+ a^+$. In particular, we have that $(ab)^+ a^+ $ is defined, which is the same as saying that $(ab)^+ c$ is defined, for all {$c \in A$}. {Thus,} $((ab)^+ c)^+  = ((ab)^+ c^+ )^+  = (ab)^+ $, for all {$c \in A$}. Therefore, { we have $c^+ = (ab)^+$, for all $c \in (ab)^+A \cup aB$} and the operation is indeed well defined.
    
    \begin{prop} \label{prop:sz-semigroupoid}
        Sz$(\cS)$ is a left restriction semigroupoid with restriction {operation} given by $(A,a)^+  = (A,a^+ )$.
    \end{prop}
    
    \begin{proof}
        Note that the existence of the composition in Sz$(\cS)$ depends only on the second coordinate of the elements. For any $(A,a), (B,b), (C,c) \in Sz(\cS)$, we have:
        \begin{align*}
            [(A,a)(B,b)](C,c) \text{ is defined } & \Leftrightarrow (A,a)[(B,b)(C,c)] \text{ is defined } \\ & \Leftrightarrow (A,a)(B,b) \text{ and } (B,b)(C,c) \text{ are defined. }
        \end{align*}
        In this case, we have:
        \begin{align*}
            [(A,a)(B,b)](C,c) &= ((ab)^+ A \cup aB,ab)(C,c) \\
            &= ((abc)^+ (ab)^+ A \cup (abc)^+ aB \cup abC,abc) \\
            &= ((abc)^+ (ab)^+ A \cup a(bc)^+ B \cup abC, abc) & \eqref{lr4} \\
            &= ((abc)^+ A \cup a(bc)^+ B \cup abC, abc) & \eqref{lema212} \\
            &= (A,a)((bc)^+ B \cup bC,bc) \\
            &= (A,a)[(B,b)(C,c)].
        \end{align*}
        This establishes that Sz\((\mathcal{S})\) is indeed a semigroupoid. To prove that $(A,a)^+ = (A,a^+)$ defines a left restriction structure on Sz\((\mathcal{S})\) we {verify conditions \eqref{lr1}, \eqref{lr2}, \eqref{lr3} and \eqref{lr4}.}\\

        \noindent\eqref{lr1} {Since $a^+a$ is defined, for all $a \in \cS$, it follows that $(A,a^+)(A,a)$ is defined and
            $$ (A,a^+)(A,a) = ((a^+a)^+ A \cup a^+A, a^+a) = (a^+A,a) = (A,a), $$
        where in the second equality we use that $a^+b = b^+b = b$, for all $b \in A$.}\\

        \noindent\eqref{lr2} {Since $a^+b^+$ is defined if and only if $b^+a^+$ is defined, we have that $(A,a^+)(B,b^+)$ is defined if and only if $(B,b^+)(A,a^+)$. Furthermore, we have
            $$ (A,a^+)(B,b^+) = ((a^+b^+)^+ A \cup a^+B, a^+b^+) = (a^+b^+(A \cup B), a^+b^+), $$
        where in the second equality we use that $a^+b^+ \in \cS^+$, hence $(a^+b^+)^+ = a^+b^+$, and that $a^+c = a^+c^+c = a^+b^+c$, for all $c \in B$. Analogously, we obtain $(B,b^+)(A,a^+) = (b^+a^+(B \cup A), b^+a^+)$. Since $a^+b^+ = b^+a^+$, it follows that $(A,a^+)(B,b^+) = (B,b^+)(A,a^+)$.}\\

        \noindent\eqref{lr3} {Suppose that $a^+b$ is defined. Then
        \begin{align*}
            ((A,a^+)(B,b))^+ &= ((a^+b)^+A \cup a^+B, (a^+b)^+) \\
            &= (a^+b^+A \cup a^+B, a^+b^+) \\
            &= (a^+b^+(A \cup B), a^+b^+) \\
            &= (A,a^+)(B,b^+),
        \end{align*}
        where in the second equality we use \eqref{lr3}, $(a^+b)^+ = a^+b^+$, and in the third and fourth equalities we use the same argument as in \eqref{lr2}.}\\

        \noindent\eqref{lr4} {Suppose that $ab$ is defined. Then
        \begin{align*}
            ((A,a)(B,b))^+ (A,a) &= ((ab)^+A \cup aB, (ab)^+)(A,a) \\
            &= ( ((ab)^+a)^+ ((ab)^+A \cup aB) \cup (ab)^+A, (ab)^+a ) \\
            &= ( (ab^+)^+(ab)^+A \cup (ab^+)^+aB \cup (ab)^+A, ab^+ ) \\
            &= ( (ab^+)^+(ab^+)^+A \cup (ab)^+aB \cup (ab^+)^+A, ab^+ ) \\
            &= ( (ab^+)^+A \cup ab^+ B, ab^+ ) \\
            &= ( (ab^+)^+A \cup aB, ab^+ ) \\
            &= (A,a)(B,b^+),
        \end{align*}
        where in the third equality we use \eqref{lr4}, $(ab)^+a = ab^+$, in the fourth equality we use Lemma \ref{lema212}, $(ab)^+ = (ab^+)^+$, and that $ee = e$, for all $e \in \cS^+$, in the fifth equality we use \eqref{lr4} again and in the sixth equality we use that $b^+B = b$.}

        {Therefore, $Sz(\cS)$ is a left restriction semigroupoid.}\\
    \end{proof}
    
The following lemma describes the natural partial order on \(\mathrm{Sz}(\cS)\), which we will also denote by \(\preceq\). Its proof is identical to that of \cite[Proposition 5.a.ii]{gomes2006generalized} and is therefore omitted.
 
    \begin{lemma}
        Let $(A,a), (B,b) \in Sz(\cS)$. Then
            $$ (A,a) \preceq (B,b) \iff a \preceq b \text{ and } a^+ B \subseteq A. $$
    \end{lemma}

 Define the map \(\iota : \cS \to \mathrm{Sz}(\cS)\) by \(\iota(s) = (\{ s^+ , s \}, s)\). Using an argument similar to the one applied in the case of left restriction semigroups \cite[Proposition 4.5]{gould2009partial}, we obtain the following expression:
 \[
 (\{s^+ , s_1, s_2, \ldots, s_n = s\}, s) = \iota(s_1)^+  \iota(s_2)^+  \cdots \iota(s_n)^+  \iota(s),
 \]
 from which the next result follows immediately.

    \begin{prop} \label{propdecszendreisemrest}
        The Szendrei expansion Sz$(\cS)$ of a left restriction semigroupoid $\cS$ is generated by the elements of the form $\iota(s)$ via composition and $^+ $.
    \end{prop}
    
We now proceed to state the principal result of this work.
    
    \begin{theorem} \label{teo_princ}
        Let $\cS$ and $\T$ be left restriction semigroupoids. If $\varphi : \cS \to \T$ is a premorphism, then there is a unique restriction morphism $\overline{\varphi} : Sz(\cS) \to \T$ such that $\varphi = \overline{\varphi} \iota$. That is, the following diagram commutes:
        \begin{center}
            \begin{tikzpicture}
                \tikzstyle{every path}=[draw, ->];
        
                \node (S) at (0,0) {$\mathcal{S}$};
                \node (Sz) at (0,-2) {$Sz(\mathcal{S})$};
                \node (T) at (2,0) {$\mathcal{T}$};
        
                \path (S) to node[above]{$\varphi$} (T);
                \path (S) to node[left]{$\iota$} (Sz);
                \path[dashed] (Sz) to node[below right]{$\overline{\varphi}$} (T);
            \end{tikzpicture}
        \end{center}
        Reciprocally, if $\overline{\varphi} : Sz(\cS) \to \T$ is a restriction morphism, then $\varphi = \overline{\varphi}\iota$ is a premorphism.
    \end{theorem}
    \begin{proof}
        Given a premorphism $\varphi : \cS \to \T$, we define $\overline{\varphi} : Sz(\cS) \to \T$ by
        \begin{align*}
            \overline{\varphi}(\{s^+ ,s_1,s_2, \ldots, s_n = s\}, s) = \varphi(s_1)^+ \varphi(s_2)^+ \cdots\varphi(s_n)^+ \varphi(s). 
        \end{align*}
        It is well defined as $s_i^+  = s^+ $ for all $1 \leq i \leq n$, and $\varphi(s_i)^+  \preceq \varphi(s_i^+ ) = \varphi(s^+ )$ for all $1 \leq i \leq n$. Hence, each $\varphi(s_i)^+ $ belongs to the same meet semilattice  component of the decomposition of {$\mathcal{T}^+$} as $\varphi(s^+ )$. Note that {$\varphi(s^+) \in \mathcal{T}^+$} by \eqref{pm2}.

     We will show that $\overline{\varphi}$ preserves $^+ $. Let $A = \{ s^+ , s_1, s_2, \ldots, s_n = s \}$. Then we have:
        \begin{align*}
            \overline{\varphi}((A,s)^+ ) = \overline{\varphi}(A,s^+ ) = \varphi(s_1)^+  \cdots \varphi(s_n)^+ \varphi(s^+ ) = \varphi(s_1)^+  \cdots \varphi(s_n)^+ ,
        \end{align*}
        since $\varphi(s_n)^+  = \varphi(s)^+  \preceq \varphi(s^+ )$. On the other hand, by Lemma \ref{lema212}, we have
        \begin{align*}
            \overline{\varphi}(A,s)^+  = (\varphi(s_1)^+  \cdots \varphi(s_n)^+ \varphi(s))^+  = (\varphi(s_1)^+  \cdots \varphi(s_n)^+ \varphi(s)^+ )^+  = \varphi(s_1)^+  \cdots \varphi(s_n)^+ ,
        \end{align*}
        from which it follows that $\overline{\varphi}((A,s)^+ ) = \overline{\varphi}(A,s)^+ $.
        
        To verify that $\overline{\varphi}$ preserves composition, consider $(A,s), (B,t) \in Sz(\cS)$ such that $(A,s)(B,t)$ is defined, where $A = \{s^+ , s_1, s_2, \ldots, s_n = s \}$ and $B = \{t^+ , t_1, \ldots, t_m = t \}$. On the one hand,
        \begin{align*}
            \overline{\varphi}(A,s)\overline{\varphi}(B,t) = \varphi(s_1)^+  \cdots \varphi(s_n)^+ \varphi(s)\varphi(t_1)^+ \cdots\varphi(t_m)^+ \varphi(t).
        \end{align*}
        On the other hand,
        \begin{align*}
            \overline{\varphi}((A,s)(B,t)) &= \overline{\varphi}((st)^+ A \cup sB, st) \\
            &= \varphi((st)^+ s_1)^+  \cdots \varphi((st)^+ s_n)^+ \varphi(st^+ )^+ \varphi(st_1)^+ \cdots\varphi(st_m)^+ \varphi(st).
        \end{align*}
        Given that $\varphi$ is a premorphism, it follows that $\varphi(st)^+  \preceq \varphi((st)^+ )$, and thus
        \begin{align*}
            \varphi((st)^+ s_i)^+ \varphi(st)^+  & = \varphi((st)^+ s_i)^+ \varphi((st)^+ )\varphi(st)^+  & (\ref{defi:premorphism}(b))\\
            & = \varphi((st)^+ )\varphi((st)^+ s_i)^+ \varphi(st)^+  & \eqref{lr2} \\
            & = \varphi((st)^+ )^+ \varphi((st)^+ s_i)^+ \varphi(st)^+  & \eqref{pm2} \\
            & = (\varphi((st)^+ )^+ \varphi((st)^+ s_i)^+ )^+ \varphi(st)^+  & \\
            & = (\varphi((st)^+ )^+ \varphi((st)^+ s_i))^+ \varphi(st)^+  & \eqref{lr3} \\
            & = (\varphi((st)^+ )\varphi(s_i))^+ \varphi(st)^+  & (\ref{defi:premorphism}(a))\\
            & = \varphi((st)^+ )\varphi(s_i)^+ \varphi(st)^+  & \eqref{lr3} \\
            & = \varphi(s_i)^+ \varphi(st)^+  & (\ref{defi:premorphism}(b)).
        \end{align*}
        Since $\varphi(st) = \varphi(st)^+ \varphi(st)$, and elements of {$\mathcal{T}^+$} commute, we may insert $n$ factors of the form $\varphi(st)^+ $ on the right-hand side of the expression
        \begin{align*}
            \overline{\varphi}((A,s)(B,t)) = \varphi((st)^+ s_1)^+  \cdots \varphi((st)^+ s_n)^+ \varphi(st^+ )^+ \varphi(st_1)^+ \cdots\varphi(st_m)^+ \varphi(st)
        \end{align*}
        and rearrange them using commutativity to obtain
\begin{align*}
            & \overline{\varphi}((A,s)(B,t)) \\
            & = \varphi((st)^+ s_1)^+ \varphi(st)^+ \varphi((st)^+ s_2)^+ \varphi(st)^+  \cdots \varphi((st)^+ s_n)^+ \varphi(st)^+ \varphi(st^+ )^+ \varphi(st_1)^+ \cdots\varphi(st_m)^+ \varphi(st).
        \end{align*}
        Using the equality $\varphi((st)^+ s_i)^+ \varphi(st)^+  = \varphi(s_i)^+ \varphi(st)^+ $, and commuting the factors $\varphi(st)^+ $ to the left of $\varphi(st)$ so that they cancel out, we obtain:
    \begin{align*}
            \overline{\varphi}((A,s)(B,t)) = \varphi(s_1)^+ \varphi(s_2)^+  \cdots \varphi(s_n)^+ \varphi(st^+ )^+ \varphi(st_1)^+ \cdots\varphi(st_m)^+ \varphi(st).
        \end{align*}
        Now, we commute the factor $\varphi(st^+ )^+ $ to the left of $\varphi(st)$ and observe that
        \begin{align*}
        \varphi(st^+ )^+ \varphi(st) \overset{\eqref{lr1}}{=} \varphi(st^+ )^+ \varphi(st^+ t) \overset{(\ref{defi:premorphism}(a))}{=} \varphi(st^+ )\varphi(t) \overset{\eqref{pm5}}{=} \varphi(s)\varphi(t).
        \end{align*}
        Hence,
        \begin{align*}
            \overline{\varphi}((A,s)(B,t)) = \varphi(s_1)^+  \cdots \varphi(s_n)^+ \varphi(st_1)^+ \cdots\varphi(st_m)^+ \varphi(s)\varphi(t).
        \end{align*}
        Since $\varphi(s)^+ \varphi(s) = \varphi(s)$ and the elements of {$\mathcal{T}^+$} commute, we now insert $m$ factors $\varphi(s)^+ $ to obtain
        \begin{align*}
            \overline{\varphi}((A,s)(B,t)) = \varphi(s_1)^+  \cdots \varphi(s_n)^+ \varphi(s)^+ \varphi(st_1)^+ \varphi(s)^+ \varphi(st_2)^+ \cdots\varphi(s)^+ \varphi(st_m)^+ \varphi(s)\varphi(t).
        \end{align*}
        Observe that $\varphi(s)^+ \varphi(st_i)^+  = (\varphi(s)^+ \varphi(st_i)^+ )^+  \overset{\eqref{lr3}}{=}(\varphi(s)^+ \varphi(st_i))^+  \overset{(\ref{defi:premorphism}(a))}{=} (\varphi(s)\varphi(t_i))^+ $. Therefore,
        \begin{align*}
            \overline{\varphi}((A,s)(B,t)) = \varphi(s_1)^+  \cdots \varphi(s_n)^+ (\varphi(s)\varphi(t_1))^+ (\varphi(s)\varphi(t_2))^+ \cdots(\varphi(s)\varphi(t_m))^+ \varphi(s)\varphi(t).
        \end{align*}
        By Lemma \ref{lema212} and \eqref{lr3}, we have $(\varphi(s)\varphi(t_i))^+ \varphi(s) = (\varphi(s)\varphi(t_i)^+ )^+ \varphi(s) = \varphi(s)\varphi(t_i)^+ $. Applying this identity iteratively, starting from $(\varphi(s)\varphi(t_m))^+ \varphi(s)$ and proceeding backward to $(\varphi(s)\varphi(t_1))^+ \varphi(s)$, we obtain
    \begin{align*}
            \overline{\varphi}((A,s)(B,t)) = \varphi(s_1)^+  \cdots \varphi(s_n)^+ \varphi(s)\varphi(t_1)^+ \varphi(t_2)^+ \cdots\varphi(t_m)^+ \varphi(t) = \overline{\varphi}(A,s)\overline{\varphi}(B,t),
        \end{align*}
        as desired. Clearly, $\varphi = \overline{\varphi}\iota$. The uniqueness of $\overline{\varphi}$ follows from Proposition \ref{propdecszendreisemrest}.
        
        Conversely, let $\overline{\varphi} : Sz(\cS) \to \T$ be a restriction morphism. Define $\varphi = \overline{\varphi}\iota$. If $st$ is defined, then $\iota(s)\iota(t) = (\{s^+ ,s\},s)(\{t^+ ,t\},t)$ is also defined. It follows that $\overline{\varphi}(\iota(s))\overline{\varphi}(\iota(t))$ is defined as well, which implies that $\varphi(s)\varphi(t)$ is defined. Hence,
        \begin{align*}
            \varphi(s)\varphi(t) & = \overline{\varphi}(\{s^+ ,s\},s)\overline{\varphi}(\{t^+ ,t\},t) = \overline{\varphi}((\{s^+ ,s\},s)(\{t^+ ,t\},t)) \\
            & = \overline{\varphi}((\{(st)^+ s^+ ,(st)^+ s, st^+ ,st\},st)) = \overline{\varphi}((\{(st)^+ , (st)^+ s, st^+ , st\}, st)),
        \end{align*}
        because $(st)^+  \preceq s^+ $. We can now apply Lemma \ref{lema212} together with \eqref{lr3} to write $(st)^+ s = (st^+ )^+ s = st^+ $. Consequently, $\varphi(s)\varphi(t) = \overline{\varphi}((\{(st)^+ , st^+ , st \}, st))$. On the other hand,
        \begin{align*}
            \varphi(s)^+ \varphi(st) & = \overline{\varphi}(\iota(s))^+ \overline{\varphi}(\iota(st)) = \overline{\varphi}(\{s^+ ,s\},s)^+ \overline{\varphi}(\{(st)^+ ,st\},st) \\
            & = \overline{\varphi}((\{s^+ ,s\},s)^+ )\overline{\varphi}(\{(st)^+ ,st\},st) = \overline{\varphi}(\{s^+ ,s\},s^+ )\overline{\varphi}(\{(st)^+ ,st\},st) \\
            & = \overline{\varphi}((\{s^+ ,s\},s^+ )(\{(st)^+ ,st\},st)) = \overline{\varphi}(\{(s^+ st)^+ s^+ ,(s^+ st)^+ s, s^+ (st)^+ ,s^+ st\},s^+ st).
        \end{align*}
        However, by \eqref{lr1} and Lemma \ref{lema212}, it follows that  $(s^+ st)^+ s^+  = (st)^+ s^+  = (st)^+ $, and from \eqref{lr1} and \eqref{lr4} it follows that $(s^+ st)^+ s = (st)^+ s = st^+ $. Thereby, we obtain
            $$ \varphi(s)^+ \varphi(st) = \overline{\varphi}(\{s^+ ,st^+ , st\},st) = \varphi(s)\varphi(t). $$
        
        Finally, notice that
        \begin{align*}
            \varphi(s)^+  & = \overline{\varphi}(\{s^+ ,s\},s^+ ) = \overline{\varphi}((\{s^+ ,s\},s^+ )(\{s^+ \},s^+ )) \\
            & =  \overline{\varphi}(\{s^+ ,s\},s^+ )\overline{\varphi}(\{s^+ \},s^+ ) =  \overline{\varphi}(\{s^+ ,s\},s^+ )\varphi(s^+ ).
        \end{align*}
        Since $\overline{\varphi}(\{s^+ ,s\},s^+ ) = \overline{\varphi}(\iota(s)^+ ) = \overline{\varphi}(\iota(s))^+  \in F$, it follows that $\varphi(s)^+  \preceq \varphi(s^+ )$, as desired.
    \end{proof}

   The Szendrei expansion $Sz(\cS)$ of a left restriction semigroupoid $\cS$ and Theorem \ref{teo_princ} generalize \cite[Theorem 2.4]{lawson2004actions} for groups, \cite[Theorem 6.17]{lawson2006expansions} for inverse semigroups, and \cite[Theorem 5.2]{gould2009partial} for left restriction semigroups. In the latter, restriction morphisms are often called (2,1)-morphisms. In the context of groupoids, partial (global) actions of a groupoid $\mathcal{G}$ on a set $X$ correspond to premorphisms (restriction morphisms) $\mathcal{G} \to \mathcal{PT}(X)$. Hence, Theorem \ref{teo_princ} provides a generalization for \cite[Theorem 3.3]{tamusiunas2023inverse}.

   {We also note that Theorem \ref{teo_princ} indirectly generalizes \cite[Proposition 4.6]{gilbert2005actions} for ordered groupoids and \cite[Theorem 5.1]{gould2011actions} for inductive constellations. More precisely, in \cite{lrspgesn} we construct a bijective correspondence between the class of restriction semigroupoids and the class of locally inductive constellations. Applying this correspondence to $Sz(\cS)$ (see \cite[Proposition 5.1]{lrspgesn}) we obtain the structure described by Gilbert \cite[Proposition 3.1]{gilbert2005actions} for ordered groupoids and by V. Gould and C. Hollings \cite[Definition 4.1]{gould2011actions} for ordered constellations.}

\subsection{The particular case of restriction categories}

    To date, no analogous correspondence to that of Theorem \ref{teo_princ} has been established in the context of restriction categories. In this subsection, we analyze what happens in this particular context. 
    
    The notion of right restriction category was introduced in \cite[Definition 2.1.1]{cockett2002restriction}. Analogously, one can define \emph{left restriction categories}, which constitute a particular case of left restriction semigroupoids. In the following, we specify Proposition \ref{prop:sz-semigroupoid} and Theorem \ref{teo_princ} for a left restriction category $\mathcal{C} = (\mathcal{C}_0,\mathcal{C}_1,D,R,\circ,+)$.

    \begin{lemma} \label{lema:lrcategory}
        In a left restriction category, all identity morphisms $1_e$ satisfy $1_{e}^{+} = 1_e$, for $e \in \mathcal{C}_0$.

        \begin{proof}
            By \eqref{lr1} and the fact that $1_e$ is a left and right identity, we obtain $1_e = 1_{e}^{+} 1_e = 1_{e}^{+}$.
        \end{proof}
    \end{lemma}
    
    \begin{prop} \label{prop:szcategory}
       The set \( Sz(\mathcal{C}) \) is a left restriction category with object set \( Sz(\mathcal{C})_0 = \mathcal{C}_0 \). For every morphism \( (A, s) \in Sz(\mathcal{C}) \), the domain and range are given by \( D(A, s) = D(s) \) and \( R(A, s) = R(s) \), respectively. Moreover, the identity morphism on an object \( e \in Sz(\mathcal{C})_0 \) is \( u_e = (\{1_e\}, 1_e) \).

        \begin{proof}
            It follows from Proposition \ref{prop:sz-semigroupoid} that $Sz(\mathcal{C})$ is a left restriction semigroupoid. Therefore, it remains to verify that $D,R \colon Sz(\mathcal{C}) \to \mathcal{C}_0$ define a directed graph structure on $Sz(\mathcal{C})$ compatible with composition, and that each $u_e$ is a (left and right) identity in $Sz(\mathcal{C})$. For $(A,s),(B,t) \in Sz(\mathcal{C})$, we have
            \begin{align*}
                (A,s)(B,t) \text{ is defined} &\iff st \text{ is defined} \\
                &\iff D(s) = R(t) \\
                &\iff D(A,s) = R(B,t).
            \end{align*}
            Furthermore, if $(A,s)(B,t)$ is defined, then
            \begin{align*}
                D((A,s)(B,t)) = D((st)^+ A \cup sB,st) = D(st) = D(t) = D(B,t),
            \end{align*}
            and
            \begin{align*}
                R((A,s)(B,t)) = R((st)^+ A \cup sB,st) = R(st) = R(s) = R(A,s).
            \end{align*}
            Therefore, $D,R \colon Sz(\mathcal{C}) \to \mathcal{C}_0$ define a directed graph structure on $Sz(\mathcal{C})$ that is compatible with composition. Let $e \in Sz(\mathcal{C})_0 = \mathcal{C}_0$. By Lemma \ref{lema:lrcategory}, we have $1_e = 1_{e}^{+}$ in $\mathcal{C}$. Hence,
                $$ u_e = (\{1_e\},1_e) = (\{1_e,1_{e}^{+}\},1_e) = \iota(1_e) \in Sz(\mathcal{C}). $$
            Since $D(u_e) = e = R(u_e)$, the composition $u_eu_e$ is defined. Suppose that $(A,s)u_e$ and $u_e(B,t)$ are defined. Since $s^+ A = A$ and $s \in A$, it follows that
            \begin{align*}
                (A,s)u_e = ( (s1_e)^+ A \cup s\{1_e\}, s1_e ) = (s^+ A \cup \{s\}, s) = (A,s).
            \end{align*}
            On the other hand, since $t^+ \in B$, we have
            \begin{align*}
                u_e(B,t) = ((1_et)^+ \{1_e\} \cup 1_eB, 1_et) = (\{t^+\} \cup B,t) = (B,t).
            \end{align*}
            This shows that each $u_e$ is an identity in $Sz(\mathcal{C})$, concluding that $Sz(\mathcal{C})$ is a category.
        \end{proof}
    \end{prop}

    From Proposition \ref{lrcategorical}, it follows that any left restriction semigroupoid $\mathcal{T}$ can be embedded into a category $\mathcal{D}$ by adjoining identity elements. Moreover, by Lemma \ref{lema:lrcategory}, the left restriction structure on $\mathcal{T}$ extends uniquely to a left restriction structure on $\mathcal{D}$ by setting $1_{e}^{+} = 1_e$.  Consequently, we may assume that every (pre)morphism $\varphi \colon \mathcal{C} \to \mathcal{T}$ has a left restriction category as its codomain.

    \begin{defi}
        A function $\varphi \colon \mathcal{C} \to \mathcal{D}$ between categories is called \textit{unitary} if $\varphi(1_e) = 1_{\varphi(e)}$ for every $e \in \mathcal{C}_0$. In particular, a \emph{unitary premorphism} (resp., \emph{unitary restriction morphism}) is a premorphism (resp., restriction morphism) between left restriction categories that is unitary.
    \end{defi}

    \begin{obs}
        Let $\mathcal{C}$ and $\mathcal{D}$ be categories.
        \begin{enumerate}
            \item Functors $\mathcal{C} \to \mathcal{D}$ are precisely unitary semigroupoid morphisms $\mathcal{C} \to \mathcal{D}$. Consequently, the restriction functors in \cite{cockett2002restriction} coincide with unitary restriction morphisms between categories.

            \item If $\mathcal{C}$ is a restriction category and $\mathcal{D}$ is equipped with the trivial left restriction ($s^+ = R(s)$, for all $s \in \mathcal{D}$), then every restriction morphism $\varphi \colon \mathcal{C} \to \mathcal{D}$ is unitary. In fact, in this case we have
                $$ \varphi(1_e) \overset{\eqref{lema:lrcategory}}{=} \varphi(1_e^+) = \varphi(1_e)^+ = R(\varphi(1_e)) = 1_{\varphi(e)}, \quad \forall e \in C_0. $$
            However, in general, restriction morphisms need not be unitary. For instance, let $M$ and $N$ be restriction monoids and $\varphi \colon M \to N$ be a unitary morphism. If we adjoin a new identity element $e$ to $N$, then $\varphi \colon M \to N \cup \{e\}$ is a restriction morphism which fails to be unitary.

            \item Unitary premorphisms between left restriction monoids coincide precisely with the unitary strong premorphisms introduced in \cite[Definition 3.7]{gould2009restriction}.
        \end{enumerate}
    \end{obs}


    \begin{theorem} \label{prop:sz-universal-cat}
        Let $\mathcal{C}$ and $\mathcal{D}$ be left restriction categories. If $\varphi \colon \mathcal{C} \to \mathcal{T}$ is a premorphism, then there is a unique restriction morphism $\overline{\varphi} \colon Sz(\mathcal{C}) \to \mathcal{D}$ such that $\varphi = \overline{\varphi} \iota$. That is, the following diagram commutes:
        \begin{center}
            \begin{tikzpicture}
                \tikzstyle{every path}=[draw, ->];
        
                \node (S) at (0,0) {$\mathcal{C}$};
                \node (Sz) at (0,-2) {$Sz(\mathcal{C})$};
                \node (T) at (2,0) {$\mathcal{D}$};
        
                \path (S) to node[above]{$\varphi$} (T);
                \path (S) to node[left]{$\iota$} (Sz);
                \path[dashed] (Sz) to node[below right]{$\overline{\varphi}$} (T);
            \end{tikzpicture}
        \end{center}
        Reciprocally, if $\overline{\varphi} \colon Sz(\mathcal{C}) \to \mathcal{D}$ is a restriction morphism, then $\varphi  = \overline{\varphi} \iota$ is a premorphism. 
        
        Moreover, $\varphi$ is a unitary premorphism if and only if $\overline{\varphi}$ is a unitary restriction morphism.

        \begin{proof}
            By Theorem \ref{teo_princ}, we have a bijective correspondence between premorphisms $\varphi \colon \mathcal{C} \to \mathcal{D}$ and restriction morphisms $\overline{\varphi} \colon Sz(\mathcal{C}) \to \mathcal{D}$. This correspondence associates to every premorphism $\varphi$ the restriction morphism $\overline{\varphi}$ defined by
                $$ \overline{\varphi}(\{s^+,s_1,\dots,s_n=s\},s) = \varphi(s_1)^+ \dots \varphi(s_n)^+ \varphi(s), $$
            and to every restriction morphism $\overline{\varphi}$ the premorphism $\varphi = \overline{\varphi} \iota$. Therefore, it remains to prove that $\varphi$ is unitary if and only if $\overline{\varphi}$ is unitary.

            Suppose that $\varphi \colon \mathcal{C} \to \mathcal{D}$ is a unitary premorphism. By Proposition \ref{prop:szcategory}, we have $Sz(\mathcal{C})_0 = \mathcal{C}_0$ and $u_e = (\{1_e\},1_e)$, for every $e \in Sz(\mathcal{C})_0$. In particular, we can identify $\overline{\varphi} \colon Sz(\mathcal{C})_0 \to \mathcal{D}_0$ with $\varphi \colon \mathcal{C}_0 \to \mathcal{D}_0$. Therefore, for every $e \in Sz(\mathcal{C})_0$, we have
                $$ \overline{\varphi}(u_e) = \varphi(1_e)^+ \varphi(1_e) = \varphi(1_e) = 1_{\varphi(e)} = 1_{\overline{\varphi}(e)}. $$
            This proves that $\overline{\varphi}$ is a unitary restriction morphism. On the other hand, assume that $\overline{\varphi} \colon Sz(\mathcal{C}) \to \mathcal{D}$ is a unitary restriction morphism. Then, for every $e \in \mathcal{C}_0$, we have
                $$ \varphi(1_e) = \overline{\varphi}(\iota(1_e)) = \overline{\varphi}( \{1_{e}^{+},1_e\},1_e ) = \overline{\varphi}(u_e) = 1_{\overline{\varphi}(e)} = 1_{\varphi(e)}. $$
            This proves that $\varphi$ is a unitary premorphism, which completes the proof.
        \end{proof}
    \end{theorem}

    Let $M$ be a monoid. Then $M$ can be equipped with a trivial left restriction structure, defined by $m^+ = 1$, for every $m \in M$. In this case, for a left restriction monoid $N$, the unitary premorphisms $M \to N$ correspond precisely to the strong premorphisms defined in \cite[Definition 2.9]{hollings2007monoids}. In the latter, the unitary restriction morphisms are referred to as $(2,1,0)$-morphisms. Thus, by applying Theorem \ref{prop:sz-universal-cat} to unitary premorphisms and restriction morphisms, we obtain the correspondence in \cite[Theorem 4.1]{hollings2007monoids} for monoids.\\

\section{A final remark}

    We end this paper proving that the Szendrei expansion $Sz(\cS)$ of a left restriction semigroupoid $\cS$ is indeed an expansion, in the sense of \cite[\textsection 2]{bierget1984expansions}. In the following, we denote by $rSGPD$ the category of left restriction semigroupoids, whose objects are left restriction semigroupoids and whose morphisms are the restriction morphisms.

    \begin{defi}
        An \emph{expansion} of left restriction semigroupoids is a pair $(F,\eta)$, where $F \colon rSGPD \to rSGPD$ is a functor and $\eta \colon F \to 1_{rSGPD}$ is a natural transformation such that $\eta_{\cS} \colon F(\cS) \to \cS$ is surjective, for every object $\cS \in rSGPD_0$.
    \end{defi}

    More precisely, to each left restriction semigroupoid $\cS$, an expansion assigns a pair $(F(\cS),\eta_{\cS})$ and, to each restriction morphism $f \colon \cS \to \mathcal{T}$, it assigns a restriction morphism $Ff \colon F(\cS) \to F(\mathcal{T})$ such that the following diagram commutes:  
    \begin{center}
        \begin{tikzpicture}
            \tikzstyle{every path}=[draw,->];

            \node (S) at (0,0) {$\cS$};
            \node (T) at (2,0) {$\mathcal{T}$};
            \node (FS) at (0,2) {$F(\cS)$};
            \node (FT) at (2,2) {$F(\mathcal{T})$};

            \path (S) to node[below]{$f$} (T);
            \path (FS) to node[left]{$\eta_{\cS}$} (S);
            \path (FT) to node[right]{$\eta_{\mathcal{T}}$} (T);
            \path (FS) to node[above]{$Ff$} (FT);
        \end{tikzpicture}
    \end{center}
    Furthermore, the assignment $f \mapsto Ff$ is functorial, in the sense that $F1_{\cS} = 1_{F(\cS)}$ for every $\cS \in rSGPD_0$, and $Fg \circ Ff = F(g \circ f)$ whenever $f \colon \cS \to \mathcal{T}$ and $g \colon \mathcal{T} \to \mathcal{L}$ are morphisms in $rSGPD$.

    \begin{lemma}
        Let $f \colon \cS \to \mathcal{T}$ be a restriction morphism. Then $Sz(f) \colon Sz(\cS) \to Sz(\mathcal{T})$, defined by $Sz(f)(A,s) = (f(A),f(s))$, where $f(A) = \{ f(a) \colon a \in A \}$, is a restriction morphism.

        \begin{proof}
            The function $Sz(f) \colon Sz(\cS) \to Sz(\mathcal{T})$ is well defined. In fact, let $(A,s) \in Sz(\cS)$. Since $f$ is a restriction morphism and $a^+ = s^+$ for every $a \in A$, we have
                $$ f(a)^+ = f(a^+) = f(s^+) = f(s)^+, \ \forall a \in A. $$
            Since $s,s^+ \in A$, it follows that $f(s),f(s)^+ \in f(A)$. Therefore, $(f(A),f(s)) \in Sz(\mathcal{T})$. Now, we verify that the function $Sz(f)$ is a semigroupoid morphism. Let $(A,s),(B,t) \in Sz(\cS)$ such that $(A,s)(B,t)$ is defined. Then $st$ is defined and, hence, both $f(s)f(t)$ and $(f(A),f(s))(f(B),f(t))$ are defined. Furthermore,
            \begin{align*}
                Sz(f)( (A,s)(B,t) ) &= Sz(f)( (st)^+ A \cup sB, st ) \\
                &= ( f((st)^+ A \cup sB), f(st) ) \\
                &= ( f(s)^+ f(t)^+ f(A) \cup f(s)f(B), f(s)f(t) ) \\
                &= (f(A),f(s))(f(B),f(t)) \\
                &= Sz(f)(A,s) Sz(f)(B,t).
            \end{align*}
            At last, $Sz(f)$ is a restriction morphism, since
            \begin{align*}
                f((A,s)^+) = f(A,s^+) = (f(A),f(s^+)) = (f(A),f(s)^+) = (f(A),f(s))^+ = Sz(f)(A,s)^+.
            \end{align*}
        \end{proof}
    \end{lemma}

    \begin{lemma}
        Define $\eta_{\cS} \colon Sz(\cS) \to \cS$ by $\eta_{\cS}(A,s) = s$, for every $(A,s) \in Sz(\cS)$. Then $\eta_{\cS}$ is a surjective restriction morphism.

         \begin{proof}
             Since $(\{s,s^+\},s) \in Sz(\cS)$, for every $s \in \cS$, it is clear that $\eta_{\cS}$ is surjective. Since $(A,s)(B,t)$ is defined in $Sz(\cS)$ if and only if $st$ is defined in $\cS$, and in this case $(A,s)(B,t) = ((st)^+ A \cup sB,st)$, it follows that
                $$ \eta_{\cS}((A,s)(B,t)) = st = \eta_{\cS}(A,s) \eta_{\cS}(B,t). $$
            Therefore, $\eta_{\cS}$ is a semigroupoid morphism. At last, since the restriction structure in $Sz(\cS)$ is given by $(A,s)^+ = (A,s^+)$, it follows that
                $$ \eta_{\cS}((A,s)^+) = s^+ = \eta_{\cS}(A,s)^+. $$
            Hence, $\eta_{\cS} \colon Sz(\cS) \to \cS$ is a restriction morphism.
         \end{proof}
    \end{lemma}

    \begin{prop}
        The pair $(Sz,\eta)$ is an expansion of left restriction semigroupoids.

        \begin{proof}
            First, we verify that $f \mapsto Sz(f)$ is functorial. It is easy to see that $Sz(1_{\cS}) = 1_{Sz(\cS)}$, for every left restriction semigroupoid $\cS$. Let $f \colon \cS \to \mathcal{T}$ and $g \colon \mathcal{T} \to \mathcal{L}$ be restriction morphisms. Then, for every $(A,s) \in Sz(\cS)$, we have
            \begin{align*}
                (Sz(g) \circ Sz(f))(A,s) = (g(f(A)),g(f(s))) = Sz(g \circ f)(A,s).
            \end{align*}
            Therefore, $Sz(g) \circ Sz(f) = Sz(g \circ f)$. It remains to prove that $\eta \colon Sz \to 1_{rSGPD}$ is a natural transformation. Let $f \colon \cS \to \mathcal{T}$ be a restriction morphism and $(A,s) \in Sz(\cS)$. Then
            \begin{align*}
                (f \circ \eta_{\cS})(A,s) = f(s) = \eta_\mathcal{T}(f(A),f(s)) = (\eta_{\mathcal{T}} \circ Sz(f))(A,s).
            \end{align*}
            That is, $f \circ \eta_{\cS} = \eta_{\mathcal{T}} \circ Sz(f)$. This concludes that $\eta \colon Sz \to 1_{rSGPD}$ is a natural transformation and, since each $\eta_{\cS}$ is surjective, $(Sz,\eta)$ is an expansion.
        \end{proof}
    \end{prop}
    
    \bibliographystyle{abbrvnat}
    \footnotesize{\bibliography{references}}

@phdthesis{liu2016free,
  title={Free inverse semigroupoids and their inverse subsemigroupoids},
  author={Liu, Veny},
  year={2016},
  school={UA Libraries}
}

@article{tilson1987categories,
  title={Categories as algebra: an essential ingredient in the theory of monoids},
  author={Tilson, Bret},
  journal={J. Pure Appl. Algebra},
  volume={48},
  number={1-2},
  pages={83--198},
  year={1987},
  publisher={Elsevier},
    DOI={10.1016/0022-4049(87)90108-3},
}

@article{exel2008inverse,
  title={Inverse semigroups and combinatorial {C}*-algebras},
  author={Exel, Ruy},
  journal={Bull. Braz. Math. Soc.},
  volume={39},
  pages={191--313},
  year={2008},
  publisher={Springer},
    DOI={10.1007/s00574-008-0080-7},
}

@article{cockett2002restriction,
  title={Restriction categories I: categories of partial maps},
  author={Cockett, J Robin B and Lack, Stephen},
  journal={Theor. Comput. Sci.},
  volume={270},
  pages={223--259},
  year={2002},
  publisher={Elsevier},
    DOI={10.1016/S0304-3975(00)00382-0},
}

@inproceedings{cordeiro2023etale,
  title={{\'E}tale inverse semigroupoids: elementary properties, universal constructions and duality},
  author={Cordeiro, Luiz Gustavo},
  booktitle={Semigr. Forum},
  volume={106},
  pages={67--127},
  year={2023},
  organization={Springer},
    DOI={10.1007/s00233-022-10329-8},
}

@article{bierget1984expansions,
    author={Birget, J and Rhodes, J},
    title={Almost finite expansions of arbitrary semigroups},
    journal={J. Pure Appl. Algebra},
    year={1984},
    volume={32},
    pages={239--287},
    DOI={10.1016/0022-4049(84)90092-6},
}

@article{szendrei1989note,
  title={A note on {B}irget-{R}hodes expansion of groups},
  author={Szendrei, M{\'a}ria B},
  journal={J. Pure Appl. Algebra},
  volume={58},
  number={1},
  pages={93--99},
  year={1989},
  publisher={Elsevier},
    DOI={10.1016/0022-4049(89)90054-6},
}

@article{exel1998partial,
  title={Partial actions of groups and actions of inverse semigroups},
  author={Exel, R},
  journal={Proc. Am. Math. Soc.},
  volume={126},
  number={12},
  pages={3481--3494},
  year={1998},
    DOI={10.1090/S0002-9939-98-04575-4},
}

@article{hollings2007monoids,
    author = {Hollings, Christopher},
    title = {Partial actions of monoids},
    journal = {Semigr. Forum},
    no = {75},
    pages = {293–316},
    year = {2007},
    DOI = {10.1007/s00233-006-0665-7},
}

@article{gould2009restriction,
  title={Restriction semigroups and inductive constellations},
  author={Gould, Victoria and Hollings, Christopher},
  journal={Comm. Algebra},
  volume={38},
  pages={261--287},
  year={2009},
  publisher={Taylor \& Francis},
    DOI={10.1080/00927870902887096},
}

@inproceedings{gould2011actions,
  title={Actions and partial actions of inductive constellations},
  author={Gould, Victoria and Hollings, Christopher},
  booktitle={Semigr. Forum},
  volume={82},
  pages={35--60},
  year={2011},
  organization={Springer},
    DOI={10.1007/s00233-010-9279-1},
}

@article{gilbert2005actions,
  title={Actions and expansions of ordered groupoids},
  author={Gilbert, ND},
  journal={J. Pure Appl. Algebra},
  volume={198},
  number={1-3},
  pages={175--195},
  year={2005},
  publisher={Elsevier},
    DOI={10.1016/j.jpaa.2004.11.006},
}

@article{gould2009partial,
  title={Partial actions of inverse and weakly left {E}-ample semigroups},
  author={Gould, Victoria and Hollings, Christopher},
  journal={J. Aust. Math. Soc.},
  volume={86},
  pages={355--377},
  year={2009},
  publisher={Cambridge University Press},
    DOI={10.1017/S1446788708000542},
}

@inproceedings{gomes2006generalized,
  title={The generalized prefix expansion of a weakly left ample semigroup},
  author={Gomes, Gracinda MS},
  booktitle={Semigr. Forum},
  volume={72},
  pages={387--403},
  year={2006},
  organization={Springer},

}

@article{exel2011semigroupoid,
  title={Semigroupoid {C}*-algebras},
  author={Exel, Ruy},
  journal={J. math. anal. appl.},
  volume={377},
  number={1},
  pages={303--318},
  year={2011},
  publisher={Elsevier},
    DOI={10.1016/j.jmaa.2010.10.061},
}

@article{tamusiunas2023inverse,
  title={Inverse semigroupoid actions and representations},
  author={Lautenschlaeger, Wesley G and Tamusiunas, Thaisa},
  journal={REMAT},
  volume={9},
  year={2023}
}

@article{lawson2006expansions,
  title={Expansions of inverse semigroups},
  author={Lawson, Mark V and Margolis, Stuart W and Steinberg, Benjamin},
  journal={J. Aust. Math. Soc.},
  volume={80},
  number={2},
  pages={205--228},
  year={2006},
  publisher={Cambridge University Press},
    DOI={10.1017/S1446788700013082},
}

@article{lawson2004actions,
  title={Partial Actions of Groups},
  author={Kellendonk, Johamnes and Lawson, Mark},
  journal={Int. J. Algebra Comput.},
  volume={14},
  number={1},
  pages={87-114},
  year={2004},
    DOI={10.1142/S0218196704001657},
}

@article{hollings2009PP,
  title={From right PP monoids to restriction semigroups: a survey},
  author={Hollings, Christopher},
  journal={Eur. J. Pure Appl. Math.},
  volume={2},
  pages={21--57},
  year={2009},
    DOI={https://www.ejpam.com/index.php/ejpam/article/view/221},
}

@article{haagtamusiunas1,
    title={Partial semigroupoid actions on sets},
    author={Haag, Rafael and Tamusiunas, Tha{\'\i}sa},
    journal={Semigroup Forum},
    doi={https://doi.org/10.1007/s00233-025-10605-3},
    year={2025}
}

@article{wagner1952semigroups,
    author={Viktor Vladimirovich Wagner},
    title={Generalised Groups},
    journal={Doklady Akademii Nauk SSSR},
    volume={84},
    year = {1952},
    pages={1119-1122}
}

@article{preston1954semigroups,
    author="Gordon Bamford Preston",
    title="Representations of inverse semigroups",
    journal="J. Lond. Math. Soc.",
    year="1954",
    volume="29",
    no="4",
    pages="411-419",
    DOI="10.1112/jlms/s1-29.4.411",
}

@article{fountain1977,
    author = {John Fountain},
    title = {Right {PP} monoids with central idempotents},
    journal = {Semigroup Forum},
    year = {1977},
    volume = {13},
    pages = {229-237},
    doi = {https://doi.org/10.1007/BF02194941},
}

@article{KribsPower2004,
  title        = {Free Semigroupoid Algebras},
  author       = {Kribs, David W. and Power, Stephen C.},
  journal      = {J. Ramanujan Math. Soc.},
  volume       = {19},
  number       = {2},
  pages        = {117--159},
  year         = {2004},
}

@article{lrspgesn,
    author = {Rafael Haag and Wesley Lautenschlaeger and Thaísa Tamusiunas},
    title = {An {E}hresmann-{S}chein-{N}ambooripad-type Theorem For Left Restriction Semigroupoids},
    journal = {Semigr. Forum},
    DOI = {10.1007/s00233-026-10616-8},
    year = {2026},
}
\end{document}